\documentclass{amsart}[12pt]

\usepackage{amsmath, amsthm, amscd, amsfonts,}
\usepackage{graphicx,float}
\usepackage{amssymb}
\usepackage[cmtip,arrow]{xy}
\usepackage{pb-diagram,pb-xy}
\usepackage{comment}
\newcommand{\PP}{{\mathbb{P}}}
\newcommand{\RR}{{\mathbb{R}}}

\newcommand{\Rforce}{\mathbb{R}}    

\DeclareMathOperator{\ps}{\mathcal{P}}

\DeclareMathOperator{\len}{lh}

\DeclareMathOperator{\crit}{crit}
\DeclareMathOperator{\dom}{dom}
\DeclareMathOperator{\cf}{cf}
\DeclareMathOperator{\Col}{Col}
\DeclareMathOperator{\Add}{Add}

\DeclareMathOperator{\ZFC}{ZFC}

\DeclareMathOperator{\Ult}{Ult}

\DeclareMathOperator{\TP}{TP}
\DeclareMathOperator{\stem}{stem}

\DeclareMathOperator{\range}{range}

\def\MPB{{\mathbb{P}}}
\def\MQB{{\mathbb{Q}}}
\def\MRB{{\mathbb{R}}}
\def\MCB{{\mathbb{C}}}

\def\k{\kappa}

\def\l{\lambda}

\def\a{\alpha}

\def\b{\beta}

\newtheorem{theorem}{Theorem}[section]
\newtheorem{lemma}[theorem]{Lemma}

\newtheorem{notation}[theorem]{Notation}
\newtheorem{definition}[theorem]{Definition}

\newtheorem{remark}[theorem]{Remark}
\newtheorem{claim}[theorem]{Claim}

\numberwithin{equation}{section}

\usepackage{pb-diagram}

\def\l{\lambda}

\def\rmark{\mbox{$\rm\bf\rule{0.06em}{1.45ex}\kern-0.05em R$}}
\def\pmark{\mbox{$\rm\bf\rule{0.06em}{1.45ex}\kern-0.05em P$}}
\def\nmark{\mbox{$\rm\bf\rule{0.06em}{1.45ex}\kern-0.05em N$}}
\def\vdash{\mbox{$\rm\| \kern-0.13em -$}}

\def\l{\lambda}

\def\rmark{\mbox{$\rm\bf\rule{0.06em}{1.45ex}\kern-0.05em R$}}
\def\pmark{\mbox{$\rm\bf\rule{0.06em}{1.45ex}\kern-0.05em P$}}
\def\nmark{\mbox{$\rm\bf\rule{0.06em}{1.45ex}\kern-0.05em N$}}
\def\vdash{\mbox{$\rm\| \kern-0.13em -$}}

\title[The Tree property at all regular  even cardinals]{The Tree property at all regular even cardinals}

\author[M. Golshani ]{Mohammad Golshani }

\thanks{The author's research has been supported by a grant from IPM (No. 97030417).}

\begin{document}

\thanks{ } \maketitle




\begin{abstract}
Assuming the existence of a strong cardinal and a measurable cardinal above it, we construct a model of $\ZFC$ in which for every singular cardinal  $\delta$, $\delta$
is strong limit,  $2^\delta=\delta^{+3}$ and the tree property  at $\delta^{++}$ holds.
This answers a question of  Friedman, Honzik and  Stejskalova \cite{friedman-honzik2}.
We also produce, relative to the existence of a strong cardinal and two measurable cardinals above it, a model of $\ZFC$ in which   the tree property holds at all
 regular even cardinals.
The result answers questions of
Friedman-Halilovic \cite{friedman-halilovic} and Friedman-Honzik \cite{friedman-honzik0}.
\end{abstract}

\section{introduction}
Trees are combinatorial objects which are of great importance in contemporary set theory. Recall that for a regular cardinal $\kappa$,
a $\k$-tree  is a tree of height $\k$ all of whose levels have size less than $\k.$
A $\k$-tree is called $\k$-Aronszajn if it has no cofinal branches.

For a regular cardinal $\k$ let the tree property at $\kappa$, denoted $\TP(\k),$  be the assertion
``there are no $\kappa$-Aronszajn trees''.
The following $\ZFC$ results are know about Aronszajn trees  (see \cite{jech}).
\begin{itemize}
\item The tree property holds at $\aleph_0$ (K\"{o}nig),
\item The tree property fails at $\aleph_1$ (Aronszajn),
\item For an inaccessible cardinal $\k,$ the tree property holds at $\k$  if and only if $\k$ is  weakly compact.

\end{itemize}
 The problem of getting the tree property at successor regular cardinals bigger than $\aleph_1$ is more subtle and it is independent of $\ZFC$ (modulo some large cardinal assumptions). The major problem,
due to Magidor, is to prove the consistency of the tree property at all regular cardinals $\kappa > \aleph_1$.

In this paper, we are interested in the tree property at regular even cardinals, i.e., cardinals of the form $\k=\aleph_\a,$ where $\a$ is an even ordinal\footnote{Recall that each ordinal $\alpha$ can be written uniquely as $\alpha=2\cdot \beta+\xi$, where $\beta \leq \alpha$ is an ordinal and $\xi<2$. $\alpha$ is called even if $\xi=0$ and it is called odd if $\xi=1$.}.

First we consider the problem of getting the tree property at double successor of singular strong limit cardinals. The first result in this direction is due to Cummings and Foreman
\cite{cummings-foreman}, who produced, starting from a supercompact cardinals $\kappa$ and a weakly compact cardinal
above it,  a model of $\ZFC$ in which $\kappa$ is a singular strong limit cardinal of countable cofinality and the tree property holds at $\kappa^{++}$.
They also extended their result  for  $\k=\aleph_\omega$.  Friedman and Halilovic \cite{friedman-halilovic}  proved the same results from a cardinal $\k$ which is $H(\lambda)$-hypermeasurable for some weakly compact cardinal $\l> \k$. In \cite{gitik-tree}, Gitik produced a model of ``$\aleph_\omega$ is strong limit and the tree property holds at $\aleph_{\omega+2}$'' from optimal hypotheses.
The papers  \cite{five authors},  \cite{friedman-honzik1} and \cite{friedman-honzik2} have continued the work, where more results about the tree property at double successor
of singular strong limit cardinals of countable cofinality are obtained.

In \cite{golshani-mohammadpour}, singular cardinals of uncountable cofinality are considered, and
in it, a model is constructed in which the tree property holds at double successor of a singular strong limit cardinal of any prescribed cofinality.

In \cite{friedman-honzik2}, Friedman, Honzik and  Stejskalova  produced a model of $\ZFC$ in which $\aleph_\omega$ is strong limit, $2^{\aleph_\omega}=\aleph_{\omega+3}$ and the tree property holds at $\aleph_{\omega+2}.$ They asked if we can replace $\aleph_\omega$ by $\aleph_{\omega_1}$.
We answer their question; in fact we prove the following  global consistency result:

 \begin{theorem} \label{thm:main theorem}
 Assume $\kappa$ is an $H(\l^{++})$-hypermeasurable cardinal where  $\lambda > \kappa$ is measurable. Then there is a generic extension $W$ of $V$ in which the following hold:
 \begin{enumerate}
 \item [(a)] $\kappa$ is inaccessible.

 \item [(b)] For every singular cardinal $\delta< \kappa,$ $\delta$ is strong limit and $2^\delta=\delta^{+3}.$

 \item [(c)] For every singular cardinal $\delta< \kappa,$ the tree property at $\delta^{++}$ holds.

 \end{enumerate}
In particular the rank initial segment $W_\kappa$ of $W$ is a model of $ZFC$ in which for every singular cardinal
  $\delta$, the tree property  at $\delta^{++}$ holds  and $2^\delta=\delta^{+3}$.
 \end{theorem}
 \begin{remark}
 Given any finite $n \geq 2,$ we can replace $2^\delta=\delta^{+3}$ with $2^\delta=\delta^{+n}.$
 \end{remark}

Then we consider the problem of getting the tree property at all regular even cardinals. In  \cite{mitchell72}, Mitchell showed that starting from two weakly compact cardinals one can get the tree property at both $\aleph_{2}$ and $\aleph_4$. Starting from infinitely many weakly compact cardinals, his result can be easily extended to get the tree property
at all $\aleph_{2n}$'s, $0<n<\omega$. The problem of getting the tree property at
 $\aleph_{2n}$'s, $0<n<\omega$ and  $\aleph_{\omega+2}$ while  $\aleph_\omega$ is strong limit
has remained open.
 In \cite{friedman-honzik0}, Friedman and Honzik produced a model in which  the tree property holds at all even cardinals below $\aleph_\omega$  where $\aleph_\omega$ is strong limit and $2^{\aleph_\omega}= \aleph_{\omega+2}$.
  Unger \cite{unger1} has extended this result to get the tree property
 at all $\aleph_n$'s, $1<n<\omega$. None of these papers obtain the tree property at $\aleph_{\omega+2}$.
We address this question and prove the following, which in particular answers a question of  \cite{friedman-honzik0}:
  \begin{theorem} \label{thm:main theorem1}
 Assume $\eta> \lambda $ are measurable cardinals above $\k$ and $\kappa$ is $H(\eta)$-hypermeasurable. Then there is a generic extension $W$ of $V$ in which:
  \begin{enumerate}
 \item [(a)] $\kappa=\aleph_\omega$ is a strong limit cardinal.

 \item [(b)] $\lambda=\aleph_{\omega+2}$.

 \item [(c)] The tree property holds at all $\aleph_{2n}$'s, $0<n<\omega$ and at $\aleph_{\omega+2}$.
 \end{enumerate}
 \end{theorem}

Then, we prove the following global consistency  result, which is related to a question of Friedman and Halilovic \cite{friedman-halilovic}.
  \begin{theorem} \label{thm:main theorem2}
 Assume $\eta> \lambda $ are measurable cardinals above $\k$ and $\kappa$ is $H(\eta^+)$-hypermeasurable. Then there is a generic extension $W$ of $V$ in which:
  \begin{enumerate}
 \item [(a)] $\kappa$ is inaccessible.

 \item [(b)] The tree property holds at all regular even cardinals below $\kappa$.
 \end{enumerate}
 In particular the rank initial segment $W_\kappa$ of $W$ is a model of $ZFC$ in which the tree property holds at all regular even cardinals.
 \end{theorem}
 The paper is organized as follows. Sections
 \ref{Some preservation lemmas} and \ref{sec:preparation model} are devoted to some preliminary results. In Section  \ref{Some preservation lemmas}
 we present some preservation lemmas and in Section \ref{sec:preparation model} we review a generalization of Mitchell's forcing and prove some facts related to it.
 In Section \ref{section:double successor}, we prove Theorem \ref{thm:main theorem}.  Section \ref{section:local evens} is devoted to the proof of Theorem \ref{thm:main theorem1}
 and finally in Section \ref{section:all even} we prove Theorem \ref{thm:main theorem2}.

We assume familiarity with forcing and large cardinals. For a forcing notion $\MPB$ we use $p \leq q$
to mean $p$ gives more information than $q$, i.e., $p \Vdash$``$q \in \dot{G}$'',
where $\dot{G}$ is the canonical $\MPB$-name for the generic filter.


\section{Some preservation lemmas}
\label{Some preservation lemmas}
The standard way to construct models with the tree property in a small cardinal is to start with some large cardinal that has certain reflection properties, and collapse it to become a specific accessible cardinal.
In order to show that the tree property holds in the generic extension, we pick a name for a $\kappa$-tree, $\dot T$, and assume towards a contradiction that it is forced to be an Aronszajn tree.
Then, we use the reflection properties of the initial large cardinal and show that the fact that $\dot T$ is Aronszajn in the generic extension, implies that the restriction of $\dot T$ to some ordinal $\alpha$ is Aronszajn tree at some intermediate stage of the forcing.
Then we show that the rest of the forcing cannot add a cofinal branch to this Aronszajn tree and get a contradiction. So the standard way to construct  models of the tree property in various small cardinals  is to go through preservation lemmas that show that various forcing notions cannot add branches to certain trees.

The following three lemmas state that forcing notions with good enough chain condition
do not add branches to Aronszajn trees
\begin{lemma}[Folklore]\label{lem: < kappa c.c.}
Let $\kappa$ be a regular cardinal and let $T$ be a $\kappa$-Aronszajn tree. Let $\mathbb{P}$ be a $\eta$-c.c.\ forcing notion, with $\eta < \kappa$. Then $\mathbb{P}$ does not add a branch to $T$.
\end{lemma}
\begin{proof}
Assume by contradiction that this is not the case and let $\dot b$ be a name for a new branch. Let us define:
\[T^\prime = \{t\in T\mid \exists p\in \mathbb{P},\,p\Vdash \check{t}\in \dot b\}\]
Note that if $x\leq_T y$ and $y\in T^\prime$ then $x\in T^\prime$ as well since $\dot b$ is forced to be downward closed. Therefore $T$ agrees with $T^\prime$ about the level of elements. Moreover, for every $\alpha < \kappa$, there is some $t\in T^\prime_\alpha$ since $\dot{b}$ is forced to be unbounded.

We conclude that $T^\prime$ is a subtree of $T$ of height $\kappa$. We claim that the width of the levels of $T^\prime$ is less than $\eta$. So let $\alpha < \kappa$ and let $\mathcal{A}$ be a maximal antichain in $\mathbb{P}$ of elements that decide the value of $\dot b$ at the level $\alpha$. By the chain condition of $\mathbb{P}$, $|\mathcal{A}| < \eta$ and therefore there are less than $\eta$ possible values for $\dot b(\alpha)$. But every element from $T^\prime_\alpha$ is a potential value for $\dot b(\alpha)$, so we conclude that $|T^\prime_\alpha| < \eta$.

By a theorem of Kurepa (see \cite{kanamori} Proposition 7.9),  $T^\prime$ has a branch.
\end{proof}
The following  branch lemma is due to Kunen and Tall, (see \cite{KunenTall}).
\begin{lemma}\label{lem: knaster forcings}
Let $\kappa$ be a regular cardinal and let $T$ be a $\kappa$-Aronszajn tree. Let $\mathbb{P}$ be a $\kappa$-Knaster forcing notion. Then $\mathbb{P}$ does not add a branch to $T$.
\end{lemma}
\begin{proof}
Let $\dot b$ be a name of a cofinal branch in $T$. For every $\alpha < \kappa,$ pick a condition $p_\alpha\in\mathbb{P}$ and an element  $t_\alpha\in T_\alpha,$ such that $p_\alpha\Vdash t_\alpha\in \dot b$. By the Knaster property, there is a cofinal subset $I$ of $\kappa$,  such that $\forall \alpha, \beta\in I$ $p_\alpha$ is compatible with $p_\beta$.

Let us choose $q\leq p_\alpha, p_\beta$. $q\Vdash t_\alpha, t_\beta \in \dot b$ and therefore it forces $t_\alpha \leq_T t_\beta$. But this is a $\Delta_0$-statement about elements of $V$ so it holds in $V$ as well. In particular, $\tilde b = \{t\in T\mid \exists \alpha\in I,\,t\leq_T t_\alpha\}$ is a cofinal branch.
\end{proof}
The following lemma is due to Silver.
\begin{lemma}
Let $\k$ and $\l$ be cardinals with $\l$ regular. Suppose that
$2^\k \geq \l$, $T$ be a $\l$-tree and $\MPB$ be $\k^+$-closed. Then forcing with $\MPB$ cannot add
a new branch through $T$.
\end{lemma}

We also need the following results of Unger (see \cite{unger}).
\begin{lemma}\label{lem: preservation product cc}
 Assume $\kappa$ is a regular cardinal and $\PP$ is a forcing notion such that $\PP \times \PP$ is $\kappa$-c.c. Then forcing with $\PP$ adds no branches to $\kappa$-trees
\end{lemma}
\begin{lemma}\label{lem: preservation product cc-2}
Let $\k$ be a regular cardinal. Let $\rho \leq \mu \leq \k$ be cardinals such that $2^\rho \geq \k$ and $2^{<\rho} < \k$. Let $\MPB$ be $\mu$-c.c. forcing notion and let
$\MQB$ be $\mu$-closed forcing notion in the ground model. Let $T$ be a $\k$-tree in $V^{\MPB}$. Then in $V^{\MPB}, \MQB$ does not add new branches to $T$.
\end{lemma}

\section{ Mitchell's forcing and its properties}
\label{sec:preparation model}
In this section we present a version of Mitchell's forcing and discuss some of its properties.
The forcing is essentially the same as  Mitchell's forcing, but it  allows us to blow up the power function as well.
\begin{definition}
Assume $\alpha < \beta$ are regular cardinals and $\gamma \geq \beta$ is an ordinal. Let $\mathbb{M}(\alpha, \beta, \gamma)$ be the following  forcing for making $2^\alpha=|\gamma|$
and forcing the tree property at $\beta=\alpha^{++}$:
\begin{itemize}
\item [(a)]
A condition in $\mathbb{M}(\alpha, \beta, \gamma)$ is a pair $(p, q),$ where
\begin{enumerate}
\item $p \in \Add(\alpha, \gamma)$,
\item $\dom(q)$ is a subset of $\beta$ of size $\leq \alpha$,
\item For each $\xi \in \dom(q), 1_{\Add(\alpha,  \xi)} \Vdash$``$q(\xi) \in \dot{\Add}(\alpha^+, 1)$''.
\end{enumerate}
\item [(b)]
For $(p, q), (p', q') \in \mathbb{M}(\alpha, \beta, \gamma)$, say $(p', q') \leq (p,q)$ iff
\begin{enumerate}
\item $p' \leq_{\Add(\alpha, \gamma)} p$,
\item $\dom(q') \supseteq \dom(q)$,
\item For all $\xi \in \dom(q), 1_{\Add(\alpha, \xi)} \Vdash$``$q'(\xi) \leq_{\dot{\Add}(\alpha^+, 1)} q(\xi)$''.
\end{enumerate}
\end{itemize}
\end{definition}
In the case $\gamma=\beta$ we obtain  Mitchell's forcing.
\begin{definition}
Assume $\alpha < \beta$ are regular cardinals. The Mitchell's forcing $\mathbb{M}(\alpha, \beta)$  is defined by
$$\mathbb{M}(\alpha, \beta)=\mathbb{M}(\alpha, \beta, \beta).$$
\end{definition}
We refer to \cite{friedman-honzik2} for more discussion about the  forcing notion $\mathbb{M}(\alpha, \beta, \beta)$ and  only present some of its basic properties which are needed in this paper.
Assume $GCH$ holds and let $\alpha< \beta \leq  \gamma$ be such that $\alpha$ is regular and $\beta$ is a measurable cardinal.
\begin{lemma}
\label{basic facts on mitchell forcing}
\begin{enumerate}
\item [(a)] $\mathbb{M}(\alpha, \beta, \gamma)$ is $\alpha$-directed closed.

\item [(b)] $\mathbb{M}(\alpha, \beta, \gamma)$  is $\beta$-Knaster.

\item [(c)] In the generic extension by $\mathbb{M}(\alpha, \beta, \gamma),$ $\alpha^+$ is preserved,  $2^\alpha \geq |\gamma|$, $\beta=\alpha^{++},$ and $\TP(\beta)$ holds.
\end{enumerate}
\end{lemma}
Let $\mathbb{T}(\alpha, \beta, \gamma)$ be the term forcing notion defined by
$$\mathbb{T}(\alpha, \beta, \gamma) = \{(\emptyset, q): (\emptyset, q) \in \mathbb{M}(\alpha, \beta, \gamma)            \}.$$
\begin{lemma}
\label{more basic facts on mitchell forcing}
\begin{enumerate}
\item [(a)] $\mathbb{T}(\alpha, \beta, \gamma)$ is $\alpha^+$-closed

\item [(b)] There exists a projection
from $\Add(\alpha, \gamma) \times \mathbb{T}(\alpha, \beta, \gamma)$ onto $\mathbb{M}(\alpha, \beta, \gamma)$.

\item [(c)] $\mathbb{M}(\alpha, \beta, \gamma) \cong \Add(\alpha, \gamma) \ast \dot{\MQB}(\alpha, \beta, \gamma)$
for an $\Add(\alpha, \gamma)$-name $\dot{\MQB}(\alpha, \beta, \gamma)$ which is forced by $\Add(\alpha, \gamma)$ to be $\alpha^+$-distributive.
\end{enumerate}
\end{lemma}

The next lemma strengthens Lemma \ref{basic facts on mitchell forcing}(c).
 \begin{lemma}
 \label{baby iteration of mitchell}
 Assume $\alpha < \beta < \gamma$,
where $\alpha$ is regular and $\beta, \gamma$ are measurable cardinals. Let $G \ast H$ be an
$\mathbb{M}(\alpha, \beta) \ast \dot{\mathbb{M}}(\beta, \gamma)$-generic filter over $V$.
 Then:
\begin{enumerate}
\item [(a)] $Card^{V[G \ast H]} \cap [\alpha, \gamma]=\{\alpha, \alpha^+, \beta, \beta^+, \gamma            \}$.

\item [(b)] $V[G \ast H] \models$``$2^\a=\beta=\a^{++}$ and $2^\beta=\gamma=\beta^{++}$''.

\item [(c)]  $V[G \ast H] \models$``The tree property holds at both $\beta$ and $\gamma$''.
\end{enumerate}
 \end{lemma}
 \begin{proof} \footnote{ The proof presented here is suggested by Yair Hayut, which is based on ideas of Unger \cite{unger}}
Parts $($a$)$ and $($b$)$ can be proved easily using Lemmas \ref{basic facts on mitchell forcing} and \ref{more basic facts on mitchell forcing}. Let us prove $(c)$. By Lemma \ref{basic facts on mitchell forcing}(c), $\TP(\gamma)$ holds in $V[G \ast H],$ thus let us show that  $\TP(\beta)$ holds in $V[G \ast H]$.

Let $\dot{T}$ be an $\mathbb{M}(\alpha, \beta) \ast \dot{\mathbb{M}}(\beta, \gamma)$-name for a $\beta$-tree. We have
\[
\mathbb{M}(\alpha, \beta) \ast \dot{\mathbb{M}}(\beta, \gamma)\cong \mathbb{M}(\alpha, \beta) \ast \dot{\Add}(\beta, \gamma) \ast \dot{\mathbb{Q}}(\beta, \gamma, \gamma)
\]
where
\[
\Vdash_{\mathbb{M}(\alpha, \beta) \ast \dot{\Add}(\beta, \gamma)}\text{``}\dot{\mathbb{Q}}(\beta, \gamma, \gamma) \text{~is~}\beta^+\text{distributive''.}
\]
Thus $\dot{T}$ is added by $\mathbb{M}(\alpha, \beta) \ast \dot{\Add}(\beta, \gamma)$. By the chain condition and the homogeneity of $\Add(\beta, \gamma)$, $\dot T$ is equivalent to an $\mathbb{M}(\alpha, \beta) \ast \dot \Add(\beta, 1)$-name.

Let $j\colon V\to M$ be an elementary embedding with critical point $\beta$, and set  $\mathbb{Q} = \mathbb{M}(\alpha, \beta) \ast \dot{\Add}(\beta, 1)$.  Then
 $$j(\mathbb{Q}) = j(\mathbb{M}(\alpha, \beta)) \ast  \dot \Add(j(\beta), 1) \cong \mathbb{M}(\alpha, \beta) \ast \dot \Add(\alpha^{+}, 1) \ast \dot{\mathbb{R}} \ast \dot \Add(j(\beta), 1)$$
 for some name $\dot{\mathbb{R}}$, which is forced to be $\alpha^+$-distributive.
 Since $\Add(\alpha^{+},1)$ is forcing equivalent to $\Col(\alpha^{+}, 2^\alpha)$ and after forcing with $\mathbb{M}(\alpha, \beta)$, $2^\alpha =2^{<\beta}= \beta$, we have
 \[
 \Add(\alpha^{+},1) \cong \Col(\alpha^{+}, \beta) \cong \Add(\beta, 1) \times \Col(\alpha^{+}, \beta) \cong \Add(\beta, 1) \ast \dot \Col(\alpha^{+}, \beta)
 \]
and hence
\[
\mathbb{M}(\alpha, \beta) \ast \dot \Add(\alpha^{+}, 1) \cong \mathbb{M}(\alpha, \beta) \ast \dot \Col(\alpha^{+}, \beta) \cong \mathbb{M}(\alpha, \beta) \ast \dot \Add(\beta, 1) \ast \dot \Col(\alpha^{+}, \beta).
\]
Thus, we can represent $j(\mathbb{Q})$ in the following way:
$$j(\mathbb{Q}) \cong \mathbb{Q} \ast \dot \Col(\alpha^{+}, \beta) \ast \dot{\mathbb{R}} \ast \dot \Add(j(\beta), 1).$$

Let $J$ be the $\mathbb{Q}$-generic filter over $V$ derived from $G \ast H$.
Using the closure of $\Add(j(\beta), 1)$, one can obtain a master condition and force a generic filter $K$ over $M[J]$ such that $J \ast K$ is a generic filter for $j(\mathbb{Q})$ and $j''[J]\subseteq J \ast K$. Therefore, in $V[J \ast K]$, we can extend $j$ to an elementary embedding $\tilde{j} \colon V[J] \to M[J \ast K]$. In particular, $\tilde{j}(\dot{T}^J)$ is a $j(\beta)$-tree and thus by taking any element from the $\beta$-th level of $\tilde{j}(\dot{T}^J)$ we can obtain a branch $b$ of $\dot{T}^J$.

Let us show that the forcing $$\Col(\alpha^{+}, \beta) \ast \dot{\mathbb{R}} \ast \dot \Add(j(\beta), 1)$$
 cannot add a branch to $\dot{T}^J$, so that $b \in V[J] \subseteq V[G \ast H].$

 By Lemma \ref{more basic facts on mitchell forcing}(b), there exists a projection from
$ \Add(\alpha, \beta) \times \mathbb{T}(\alpha, \beta, \beta)$ onto $\mathbb{M}(\alpha, \beta),$
and hence using $j$, we get a projection from
$\Add(\alpha, j(\beta)) \times \mathbb{T}(\alpha, j(\beta), j(\beta))$ onto $j(\mathbb{M}(\alpha, \beta)).$
Since $j(\mathbb{M}(\alpha, \beta)) \cong \mathbb{M}(\alpha, \beta) \ast \dot \Col(\alpha^{+}, \beta) \ast \dot{\mathbb{R}}$, thus we get a projection
form $\Add(\alpha, j(\beta)) \times \mathbb{T}(\alpha, j(\beta), j(\beta))$ onto $\Col(\alpha^{+}, \beta) \ast \dot{\mathbb{R}}$. The forcing $\mathbb{T}(\alpha, j(\beta), j(\beta)) \ast \dot \Add(j(\beta), 1)$ is $\alpha^+$-closed of size $j(\beta),$ and hence
\[
\Col(\alpha^{+}, j(\beta)) \cong \Col(\alpha^{+}, j(\beta)) \times (\mathbb{T}(\alpha, j(\beta), j(\beta)) \ast \dot \Add(j(\beta), 1)).
\]
Putting all things together, we get a projection
\[
\pi: \Add(\alpha, j(\beta)) \times \Col(\alpha^{+}, j(\beta)) \to \Col(\alpha^{+}, \beta) \ast \dot{\mathbb{R}} \ast \dot \Add(j(\beta), 1).
\]
But Lemma \ref{lem: preservation product cc-2}, the forcing $ \Add(\alpha, j(\beta)) \times \Col(\alpha^{+}, j(\beta))$ cannot add a branch to $\dot{T}^J$. It follows that $\Col(\alpha^{+}, \beta) \ast \dot{\mathbb{R}} \ast \dot \Add(j(\beta), 1)$ does not add a branch to $\dot{T}^J$, as required.
\end{proof}
\begin{remark}
The lemma is true if we assume $\beta$ and $\gamma$ are  weakly compact cardinals. The argument is essentially the same, where instead of an embedding from the whole universe, we use a weakly compact embedding $j: \mathcal{M} \to \mathcal{N}$, where $\mathcal{M}$ contains all relevant
information.
\end{remark}
In a similar way, we can prove the following.
 \begin{lemma}
 \label{iteration of mitchell}
 Assume $\alpha_0 < \alpha_1 < \dots < \alpha_n$,
where $\alpha_0$ is regular and $\alpha_1, \dots, \alpha_n$ are measurable cardinals and let $G= G_0 \ast G_1 \ast \dots \ast G_{n-1}$ be a generic filter over $V$ for the forcing notion
\[
\mathbb{M}(\a_0, \a_1) \ast \dot{\mathbb{M}}(\a_1, \a_2) \ast \dots \ast \dot{\mathbb{M}}(\a_{n-1}, \a_n).
\]
 Then
\begin{enumerate}
\item [(a)] $Card^{V[G]} \cap [\alpha_0, \alpha_n]=\{\alpha_0, \alpha_0^+, \alpha_1, \alpha_1^+, \dots, \alpha_{n-1}, \alpha_{n-1}^+, \alpha_n             \}$.

\item [(b)] $V[G] \models$``For each $i<n, 2^{\a_i}=\alpha_{i+1}=\alpha_i^{++}$''.

\item [(c)]  $V[G] \models$``The tree property holds at each $\alpha_i, 1 \leq i \leq n$''.
\end{enumerate}
 \end{lemma}
Assume that
 $\eta > \l$ are measurable cardinals above $\k$
 and there exists an elementary embedding $j: V \to M$ with critical point $\k$
such that $H(\eta) \subseteq M$
and
 $j$ is generated by a $(\k, \eta)$-extender.
 Suppose there exists  $\bar g \in V$
which is $i(\Add(\k, \l)_{V})$-generic over $N,$
where $U$ is the normal measure derived from $j$ and  $i: V \to N \simeq \Ult(V, U)$
is the ultrapower embedding.

Let
\[
\MPB=\langle  \langle \MPB_\a: \a \leq \k+1      \rangle, \langle \dot{\MQB}_\a: \a \leq \k          \rangle\rangle
\]
be the reverse Easton iteration, where
\begin{enumerate}
\item If $\a \leq \k$ is a measurable limit of measurable cardinals, then
\begin{center}
$\Vdash_\a$``$\dot{\MQB}_\a =\dot{\mathbb{M}}(\a,  \a_*) \ast \dot{\mathbb{M}}(\a_*,  \a_{**})$'',
\end{center}
where for each $\a \leq \k$, $\a_* < \a_{**}$ are the first and the second measurable cardinals above $\a$ respectively.
\item Otherwise, $\Vdash_\a$``$\dot{\MQB}_\a$ is the trivial forcing''.
\end{enumerate}
Let
$$G=\langle \langle G_\a: \a \leq \k+1      \rangle, \langle  G(\a): \a \leq \k      \rangle\rangle $$
be
 $\MPB $-generic over $V$.
Also let $\mathbb{M}=\mathbb{M}(\aleph_0, \k).$
\begin{lemma}
\label{laver vs mitchell}
Forcing with $\mathbb{P} \times \mathbb{M}$ forces ``$\kappa=\aleph_2$ and $\TP(\k)$ holds''.

Similarly,  if we replace $\mathbb{P}$ with $\mathbb{P}_{(\lambda, \k+1]}$, the tail iteration after $\l$, and $\mathbb{M}$ with
$\mathbb{M}(\l, \k),$ then the resulting product forces ``$\kappa=\l^{++}$ and $\TP(\k)$ holds''
\end{lemma}
\begin{proof}
It is easily seen that the forcing notion   $\mathbb{P} \times \mathbb{M}$   preserves $\aleph_1$ and $\kappa$ and forces $\kappa=\aleph_2$.
Let us show that it forces the tree property at $\k.$

Let $G \times H$ be $\mathbb{P} \times \mathbb{M}$-generic over $V$ and assume  $T$ is a $\k$-tree in $V[G \times H]$.
We have
\[
j(\MPB)  \cong \MPB_\k \ast \dot{\mathbb{M}}(\k,  \l) \ast \dot{\mathbb{M}}(\l,  \eta) \ast j(\MPB)_{(\k+1, j(\k))} \ast \dot{\mathbb{M}}(j(\k),  j(\l)) \ast \dot{\mathbb{M}}(j(\l),  j(\eta))
\]
where $\Vdash_{\MPB_\k \ast \dot{\mathbb{M}}(\k,  \l) \ast \dot{\mathbb{M}}(\l,  \eta)}$`` $ j(\MPB)_{(\k+1, j(\k))}$ is $\kappa^+$-closed and $j(\k)$-c.c.''. On the other hand,
\[
\mathbb{M}(\k,  \l) \ast \dot{\mathbb{M}}(\l,  \eta) \cong \Add(\k, \l)\ast \dot{\MQB},
\]
where $\dot{\MQB}$ is forced to be $\k^+$-distributive. Thus
\[
j(\MPB)  \cong \MPB_\k \ast  \dot \Add(\k, \l)\ast \dot{\MQB} \ast  j(\MPB)_{(\k+1, j(\k))} \ast  \dot \Add(j(\k), j(\l))\ast j(\dot{\MQB}).
\]
Let us write $G$ as $$G=G_\k \ast g \ast h$$
which corresponds to $\MPB=\MPB_\k \ast  \dot \Add(\k, \l)\ast \dot{\MQB}.$

Let $k: N \to M$ be the induced elementary embedding so that $k \circ i=j.$
By standard arguments, we can lift  $k$ to $N[G_\k]$ to get $k: N[G_\k] \to M[G_\k].$

Let   $\bar \l $ be such that
\begin{center}
$N \models$``$\bar \lambda$ is the least measurable cardinal above $\k$''.
\end{center}
Note that $\k^+ < \bar \l < \k^{++}.$
Factor $g$ as $g_1 \times g_2$, which corresponds to
\[
\Add(\k, \l)_{V[G_\k]} = \Add(\k, \bar \l)_{V[G_\k]} \times \Add(\k, \l \setminus \bar \l)_{V[G_\k]}.
\]
We can further
extend $k$ to get $k: N[G_\k][g_1] \to M[G_\k][g].$

By an argument as above, we can write
$i(\MPB)$ as
\[
i(\MPB)  \cong \MPB_\k \ast  \dot \Add(\k, \bar \l)\ast \dot{\MQB}_1 \ast  i(\MPB)_{(\k+1, i(\k))} \ast  \dot \Add(i(\k), i(\bar \l))\ast i(\dot{\MQB}_1).
\]
where  $\Vdash_{\MPB_\k \ast  \dot \Add(\k, \bar \l)}$`` $ \dot{\MQB}_1$ is $\kappa^+$-distributive''.
Let $h_1$ be the filter generated by $i''(h)$. Then $h_1$
is $\dot\MQB_1[G_\k \ast g_1]$-generic over $N[G_\k \ast g_1]$.

By standard arguments, we can find $K \in V[G_\k \ast g_1 \ast h_1]$,
which is $i(\MPB)_{(\k+1, i(\k))}$-generic over $N[G_\k \ast g_1 \ast h_1].$
We transfer $g_1 \ast h_1 \ast K$ along $k$ to get

$\hspace{5.cm}$$i: V[G_\k] \to N[i(G_\k)],$

$\hspace{5.cm}$$k: N[i(G_\k)] \to M[j(G_\k)],$

where the maps are defined in $V[G_\k \ast g\ast h]$.
Since $\MPB_\k$ has size $\k$ and is $\k$-c.c., so the term forcing
\[
\Add(\k, \l)_{V[G_\k]} / \MPB_\k
\]
is forcing isomorphic to $\Add(\k, \l)_{V}$ (see \cite{cummings} Fact 2, $\S$1.2.6). By our assumption, we have
$\bar g \in V$
which is $i(\Add(\k, \lambda)_{V})$-generic over $N$,
and using it we can define $g_a$ which is
\[
\Add(i(\k), i(\l))_{N[i(G_\k)]}
\]
generic over $N[i(G_\k)]$.

Transfer $g_a$  along $k$ to get the new generic $\bar g_a$.
Now using Woodin's surgery argument, we can alter the filter $\bar g_a$ to find a generic filter $h_a$
with the additional property
$j''[g] \subseteq h_a$.

So we can build maps

$\hspace{4.cm}$$j: V[G_\k \ast g] \to M[j(G_\k \ast g)],$

$\hspace{4.cm}$$k: N[i(G_\k \ast g)] \to M[j(G_\k \ast g)]$,

$\hspace{4.cm}$$i: V[G_\k \ast g] \to N[ i(G_\k \ast g)].$

The forcing $\dot{\MQB}[G_\k \ast g]$  is $\k^+$-distributive in $V[G_\k \ast g]$,
so we can further extend the above embeddings  and get

$\hspace{3.cm}$$j: V[G_\k \ast g \ast h] \to M[j(G_\k \ast g \ast h)],$

$\hspace{3.cm}$$k: N[i(G_\k \ast g \ast h)] \to M[j(G_\k \ast g \ast h)]$,

$\hspace{3.cm}$$i: V[G_\k \ast g \ast h] \to N[ l(G_\k \ast g \ast h)].$

In particular, we have $j: V[G] \to M[j(G)].$ Now suppose $j(H)$ is $j(\mathbb{M})$-generic over $M[j(G)]$ such that $j$ lifts to

\[
\tilde{j}: V[G \times H] \to M[j(G) \times j(H)].
\]
Then $T \in M[G \times H]$ and $T$ has a cofinal branch  in $M[j(G) \times j(H)]$.
But note that
$$j(\mathbb{P})/G \simeq  j(\MPB)_{(\k+1, j(\k))} \ast  \dot \Add(j(\k), j(\l))\ast j(\dot{\MQB})$$
 is $\k^+$-closed. Also by
the proof of Lemma \ref{baby iteration of mitchell}, $j(\mathbb{M})/H$ can not add  a cofinal branch in $T$. Thus it follows that forcing with $j(\mathbb{P})/G \times \mathbb{M}/ H$
 can not add a branch through $T$,
 which leads to a contradiction. The lemma follows.
\end{proof}
In fact one can say more. Let $U$ be a normal measure on $\kappa.$ Then for any $\l < \k,$ one can show that
\[
\Lambda_\l= \{ \gamma \in (\l, \k):  \mathbb{P}_{(\lambda, \gamma]} \times \mathbb{M}(\l, \gamma)\Vdash~``~ \gamma=\l^{++} + TP(\gamma) \text{~''}     \} \in U.
\]
and hence $$\Lambda =\bigtriangleup_{\l < \k} \Lambda_\l =\{ \gamma< \k: \forall \l<\gamma, \gamma \in \Lambda_\l               \} \in U.$$

In a similar way, we can prove the following lemma that will be used later.
\begin{lemma}
\label{laver vs mitchell2}
Suppose $\a<\beta < \k$
 are such that $\a$ is regular and $\beta$ is measurable. Then
\[
\mathbb{M}(\a, \beta) \ast (\dot{\mathbb{P}}_{(\beta, \k+1])} \times \dot{\mathbb{M}}(\beta, \k))
\]
forces the tree property at both $\beta$ and $\k,$ where working in $V^{\mathbb{M}(\a, \beta)}$, the iteration $\MPB$ is defined as before
and $\MPB_{(\beta, \k+1])}$
denotes the tail iteration after $\beta$.
\end{lemma}

\section{The tree property at double successor of  singular cardinals}
\label{section:double successor}
In this section we prove Theorem \ref{thm:main theorem}.  In Subsection \ref{measure sequences} we define the notion of a measure sequence
and then in Subsection \ref{Radin forcing with interleaved collapses}, we assign to each measure sequence $w$ a forcing notion $\Rforce_w$, which is a  version of Radin forcing which is needed for the proof of
Theorem \ref{thm:main theorem}.
In Subsection \ref{sec: basic properties}, we review some of the basic properties of the forcing notion $\Rforce_w$.
Then in Subsection \ref{sec: final model} we define the required model  and in Subsections
\ref{sec:tree property at double successors} and \ref{sec:completing the proof} we complete the proof of Theorem \ref{thm:main theorem}.
\subsection{Measure sequences}
\label{measure sequences}
In this subsection we define a class $\mathcal{U}_\infty$ of measure sequences which are needed for the proof of Theorem \ref{thm:main theorem}.
Our presentation follows \cite{golshani-diamond}, but we present the details for completeness.
During the Subsections \ref{measure sequences}, \ref{Radin forcing with interleaved collapses} and  \ref{sec: basic properties}, we assume that the following
conditions are satisfied:
\begin{itemize}
\item $\kappa$ is an $H(\k^{++})$-hypermeasurable cardinal.

\item $2^\k=2^{\k^+}=2^{\k^{++}}=\k^{+3}$.

\item  There is $j: V \to M$ with critical point $\k$
such that $H(\k^{++}) \subseteq M$.

\item  $j$ is generated by a $(\k, \k^{+4})$-extender.

\item $\k^{+4} < j(\k) < \k^{+5}$.

\item If $U$ is the normal measure derived from $j$ and if $i: V \to N \simeq \Ult(V, U)$
is the ultrapower embedding, then there exists  $F \in V$
which is $\Col(\k^{+5}, < i(\k))_{N}$-generic over $N.$
\end{itemize}
Let   $k: N \to M$ be the induced elementary embedding with $j=k \circ i$.
Then $crit(k)=\k^{+4}_N < \k^{+4}_M = \k^{+4}$.
Set

$\hspace{1.5cm}$ $P^*=\{  f: \kappa \to V_\kappa \mid \dom(f) \in U$ and $\forall \alpha, f(\alpha) \in    \Col(\alpha^{+5}, < \kappa)                \}$.

$\hspace{1.5cm}$ $F^* = \{ f \in P^* \mid i(f)(\kappa) \in F    \}$.

Then $U$ can be read off $F^*$ as
$$U= \{ X \subseteq \kappa \mid \exists f \in F^*, X=\dom(f)    \}.$$
The following definitions are based on  \cite{cummings} and  \cite{golshani-diamond} with the modifications required for our purposes.
\begin{definition}
A constructing pair is a pair $(j, F)$, where
\begin{itemize}
\item $j: V \to M$ is a non-trivial elementary embedding into  a transitive inner model, and if $\kappa=\crit(j),$ then
$M^\kappa\subseteq M.$

\item $F$ is $\Col(\kappa^{+5}, < i(\kappa))_N$-generic over $N$, where $i: V \to N \simeq \Ult(V, U)$ is the ultrapower embedding approximating $j$. Also factor $j$ through $i$, say $j=k \circ i.$

\item $F \in M$.

\item $F$ can be transferred along $k$ to give a $\Col(\kappa^{+5}, < j(\kappa))_M$-generic over $M$.
\end{itemize}
\end{definition}
In particular note that the pair $(j, F)$ constructed above is a constructing pair.
\begin{definition}
\label{constructing pair}
If $(j, F)$ is a constructing pair as above, then $F^* = \{ f \in P^* \mid i(f)(\kappa) \in F    \}.$
\end{definition}
\begin{definition}
Suppose $(j, F)$ is a constructing pair as above. A sequence $w$ is constructed by $(j, F)$ iff
\begin{itemize}
\item $w \in M.$
\item $w(0) = \kappa = \crit(j)$.
\item $w(1)=F^*.$
\item For $1 < \beta< \len(w), w(\beta)=\{ X \subseteq V_\kappa \mid w \upharpoonright \beta \in j(X)              \}$.
\item $M \models |\len(w)| \leq w(0)^{+}$.
\end{itemize}
\end{definition}
\begin{remark}
In \cite{cummings}, it is assumed that $M \models |\len(w)| \leq w(0)^{++}$, while here we just require that $M \models |\len(w)| \leq w(0)^{+}$. This is because we only need to preserve the inaccessibility of $\k=w(0)$ and by results of Mitchell (see Lemma \ref{thm:preserving inaccessibility}) it suffices that the length of the measure sequence to be $\k^+$.
\end{remark}
If $w$ is constructed by $(j, F)$, then we set $\kappa_w=w(0),$
and if $\len(w) \geq 2$, then we define

$\hspace{1.5cm}$ $F^*_w=w(1)$.

$\hspace{1.5cm}$ $\mu_w = \{ X \subseteq \kappa_w \mid \exists f \in F^*_w, X = \dom(f)     \}$.

$\hspace{1.5cm}$ $\bar{\mu}_w = \{ X \subseteq V_{\kappa_w} \mid \{ \alpha \mid \langle \alpha \rangle \in X     \} \in \mu_w       \}$.

$\hspace{1.5cm}$ $F_w= \{ [f]_{\mu_w} \mid f \in F^*_w   \}$.

$\hspace{1.5cm}$ $\mathcal{F}_w = \bar{\mu}_w \cap \bigcap \{w(\alpha) \mid 1 < \alpha < \len(w)          \}$.

\begin{definition}
Define inductively

$\hspace{1.5cm}$ $\mathcal{U}_0 = \{ w \mid \exists (j, F)$ such that $(j, F)$  constructs $w         \}$.

$\hspace{1.5cm}$ $\mathcal{U}_{n+1}= \{ w \in \mathcal{U}_n \mid \mathcal{U}_n \cap V_{\kappa_w} \in  \mathcal{F}_w         \}.$

$\hspace{1.5cm}$ $\mathcal{U}_{\infty}= \bigcap_{n \in \omega} \mathcal{U}_{n}$.

The elements of $\mathcal{U}_{\infty}$ are called measure sequences.
\end{definition}
Let $u$ be the measure sequences constructed using the pair $(j, F)$ above. It is easily seen that for each $\alpha < \kappa^{++}, u \upharpoonright \alpha$
exists and is in $\mathcal{U}_{\infty}$.

\subsection{Radin forcing with interleaved collapses}
\label{Radin forcing with interleaved collapses}
In this subsection, we assign to each measure sequence  $w \in \mathcal{U}_{\infty}$ a forcing notion $\Rforce_w$. The forcing $\Rforce_w$ adds a club $C$ of ground model regular cardinals into $\k_w$ in such a way that if $\a < \b$ are successive points in $C$, then it collapses all cardinals in the interval $(\a^{+6}, \b)$ into $\a^{+5}$
and makes $\b=\a^{+6}$. When $\len(w)=\k_w^+,$ then the forcing preserves the inaccessibility of $\k_w$
and all singular cardinals less than $\k_w$  are limit points of $C$. As we will see in the next subsections, if we start with a suitably prepared model
and force over it by $\Rforce_w$, where $\len(w)=\k_w^+,$ then in the resulting extension the tree property holds at  double successor of every limit point of $C$ and
in particular, the rank initial segment of the final model at $\k_w$ is a model in which the tree property holds at double successor of every singular cardinal.

First we define the building blocks of the forcing.
\begin{definition}
Assume $w \in \mathcal{U}_{\infty}.$ Then $\MPB_w$ is the set of all
  tuples $p=(w, \lambda, A, H, h),$ where
\begin{enumerate}
\item $w$ is a measure sequence.

\item $\lambda < \kappa_w$.

\item $A \in \mathcal{F}_w$.

\item $H \in F^*_w$ with  $\dom(H)=\{ \kappa_v > \lambda \mid v \in A  \}$.

\item $h \in \Col(\lambda^{+5}, < \kappa_w).$
\end{enumerate}

Note that if $\len(w)=1,$ then the above tuple is of the form  $(w, \lambda, \emptyset, \emptyset, h)$, where  $\lambda < \kappa_w$ and $h \in \Col(\lambda^{+5}, < \kappa_w).$
\end{definition}
Given $p \in \MPB_w$ as above, we denote it by
\[
p=(w^p, \lambda^p, A^p, H^p, h^p).
\]
The order on $\MPB_w$ is defined as follows.
\begin{definition}
Assume $p, q \in \MPB_w.$ Then $p \leq^* q$ iff:
\begin{enumerate}
\item $w^p=w^q$.

\item $\lambda^p=\lambda^q$.

\item $A^p \subseteq A^q$.

\item For all $v \in A^p, H^p(\k_v) \leq H^q(\k_v)$.

\item $h^p \leq h^q.$
\end{enumerate}

\end{definition}

Next we define the forcing notion $\Rforce_w$.
\begin{definition}
\label{radin forcing definition}
If $w$ is a measure sequence, then $\Rforce_w$ is the set of all finite sequences
$$p= \langle  p_k \mid k \leq n     \rangle,$$
where
\begin{enumerate}
\item $p_k = (w_k, \lambda_k, A_k, H_k, h_k) \in \MPB_{w_k}$, for each $k \leq n.$

\item $w_n=w.$


\item If $k<n,$ then $\lambda_{k+1}=\kappa_{w_k}$.
\end{enumerate}
\end{definition}
Given $p \in \Rforce_w$ as above, we denote it by
\[
p= \langle  p_k \mid k \leq n^p     \rangle
\]
and call $n^p$ the length of $p$.
We also use $w^p_k$ for $w^{p_k}$ and so on (for $k \leq n^p$).
The direct extension relation $\leq^*$ is defined  on $\Rforce_w$  in the natural way:
\begin{definition}
Assume $p, q \in \Rforce_w.$ Then $p \leq^* q$ iff
\begin{enumerate}
\item $n^p=n^q.$
\item For all $k \leq n^p, p_k \leq^* q_k$ in $\MPB_{w^{p}_k}$.
\end{enumerate}
\end{definition}
The following definition is the key step towards defining the order relation $\leq$ on $\Rforce_w$
\begin{definition}
\begin{itemize}
\item [(a)] Assume $p= (w, \lambda, A, H, h) \in \MPB_w$  and $w' \in A.$ Then $\Add(p, w')$ is the condition $\langle p_0, p_1 \rangle \in \Rforce_w$ defined by
\begin{enumerate}
\item $p_0 = (w', \lambda, A \cap V_{\kappa_{w'}}, H \upharpoonright \kappa_{w'}, h)$.

\item $p_1= (w, \kappa_{w'}, A \setminus V_\eta, H \upharpoonright \dom(H) \setminus V_\eta, H(\kappa_{w'})  ),$
where $\eta = \sup \range (H(\kappa_{w'}))$.
\end{enumerate}
 In the case that this does not yield a member of $\Rforce_w$, then $\Add(p,w')$ is undefined.

\item [(b)] Suppose $p=\langle p_0,\dots,p_n \rangle\in \Rforce_w$, $k\leq n$ and $u\in A^p_k$.  Then $\Add(p,u)$ is the member of $\Rforce_w$ obtained by replacing $p_k$ with the two members of $\Add(p_k,u)$,
That is,
\begin{itemize}
\item  $\Add(p,u) \upharpoonright  i=p \upharpoonright i.$
\item  $\Add(p,u)_{i}=\Add(p_i,u)_{0}.$
\item  $\Add(p,u)_{i+1}=\Add(p_i,u)_{1}.$
\item  $\Add(p,u) \upharpoonright [i+2,n+1]= p \upharpoonright [i+1,n].$
\end{itemize}
\end{itemize}
\end{definition}
The next lemma shows that any condition $p$  has a direct extension $q$ such that
$\Add(q, u)$ is well-defined for all $k \leq n^q$ and all $u \in A^q_k$.
\begin{lemma}
\label{almost all can be added}
\begin{itemize}
\item [(a)] Suppose $p= (w, \lambda, A, H, h) \in \MPB_w$. Then
\[
A'=\{  w' \in A:  \Add(p, w') \text{~is well-defined~}           \} \in \mathcal{F}_w.
\]
\item [(b)] Suppose $p\in \Rforce_w.$ Then there exists
 $q \leq^* p$ such that for all $k \leq n^q$ and all $u \in A^q_k, \Add(q, u) \in \Rforce_w$ is well-defined.
\end{itemize}
\end{lemma}
\begin{proof}
Clause (b)   follows from (a), so let us prove (a). Let $p= (w, \lambda, A, H, h) \in \MPB_w$. We have to show that
$A' \in \bar{\mu}_w \cap \bigcap \{w(\alpha) \mid 1 < \alpha < \len(w)          \}.$

Suppose $(j, F)$ constructs $w$, where $j: V \to M,$ and let $i: V \to N \simeq \Ult(V, U)$
be the corresponding ultrapower embedding.

Let us first show that $A' \in \bar{\mu}_w.$ We have

$\hspace{2.cm}$ $A' \in \bar{\mu}_w \iff \{ \a: \langle \a \rangle  \in A'           \} \in U$

$\hspace{3.3cm}$ $\iff \k_w \in j(\{ \a: \langle \a \rangle  \in A'           \})$

$\hspace{3.3cm}$ $\iff \langle \k_w \rangle  \in j(A')$

$\hspace{3.3cm}$ $\iff \Add(j(p), \langle \k_w \rangle)$ is a well-defined condition in $\Rforce^M_{j(w)}.$

But we have
$$ \Add(j(p), \langle \k_w \rangle)= \langle (\langle \k_w \rangle, \l, A, H, h), (j(w), \k_w, A^*, H^*,  j(H)(\k_w))                        \rangle,$$
where $A^*=j(A) \setminus V_{\k_w}$ and $H^*= j(H) \upharpoonright \dom(j(H)) \setminus V_{\k_w}$. Thus $\Add(j(p), \langle \k_w \rangle)$   is  well-defined,
which implies $A' \in \bar{\mu}_w.$

Now let  $1< \a < \len(w).$  Then

$\hspace{2.cm}$ $A' \in w(\a) \iff w \upharpoonright \a \in j(A')$

$\hspace{3.65cm}$ $\iff \Add(j(p), w \upharpoonright \a)$ is a well-defined condition in $\Rforce^M_{j(w)}.$

By an argument as above, it is easily seen that $\Add(j(p), w \upharpoonright \a)$ is a well-defined condition in $\Rforce^M_{j(w)},$ and hence $A' \in w(\a)$.
Thus $$A' \in \bar{\mu}_w \cap \bigcap \{w(\alpha) \mid 1 < \alpha < \len(w)          \}$$ as required.
\end{proof}
By the above lemma we can always assume that $\Add(q, u)$ is well-defined for all $q \in \Rforce_w$, all $k \leq n^q$ and all $u \in A^q_k$.
\begin{definition}
\begin{itemize}
\item [(a)]  Suppose $p, q \in \Rforce_w$. Then $p$ is a one-point extension of $q$, denoted $p \leq^1 q$, if there are $k \leq n^q$ and $u \in A^q_k$
such that $p \leq^* \Add(q, u).$

\item [(b)] Suppose $p, q \in \Rforce_w$. Then $p$ is an extension of $q$, denoted $p \leq q$, if there are $n< \omega$
and conditions $p_0, p_1, \dots, p_n$ such that
$$p=p_0 \leq^1 p_1 \leq^1 \dots \leq^1 p_n=q.$$
\end{itemize}
\end{definition}

\subsection{Basic properties of the forcing notion $\Rforce_{w}$}
\label{sec: basic properties}
We now state and prove some of the main properties of the forcing notion $\Rforce_w$.
\begin{lemma} \label{chain condition for radin forcing}
$(\Rforce_w, \leq)$ satisfies the $\kappa_w^+$-c.c.
\end{lemma}
\begin{proof}
Assume on the contrary that $A \subseteq \Rforce_w$ is an antichain of size $\kappa_w^+.$ We can assume that all $p \in A$ have the same length $n$.
Write each $p \in A$ as $p = d_p ^{\frown} p_n,$ where $d_p \in V_{\kappa_w}$ and $p_n = (w, \lambda^p, A^p, H^p, h^p)$. By shrinking $A$, if necessary,
we can assume that there are fixed $d \in V_{\kappa_w}$ and $\lambda < \kappa_w$ such that for all $p \in A,$ $d_p=d$ and $\lambda^p = \lambda.$

Note that for $p \neq q$ in $A$, as $p$ and $q$ are incompatible, we must have $h^p$ is incompatible with $h^q$. But $\Col(\lambda^{+5}, < \kappa_w)$
satisfies the $\kappa_w$-c.c., and we get a contradiction.
\end{proof}
The following factorization lemma can be proved easily.
\begin{lemma} (The factorization lemma)  \label{thm:factorization lemma}
Assume that $p=\langle p_0,\dots,p_n \rangle\in \Rforce_w$, where $p_i=(w_i, \lambda_i, A_i, H_i, h_i)$ and $m<n.$ Set $p^{\leq m}= \langle p_0, \dots, p_m \rangle$
and $p^{>m} = \langle  p_{m+1}, \dots, p_n     \rangle$.
Then

$($a$)$  $p^{\leq m} \in \Rforce_{w_m},$
$p^{>m} \in \Rforce_{w}$ and there exists
$$i: \Rforce_w / p \rightarrow \Rforce_{w_m}/p^{\leq m} \times \Rforce_w / p^{>m}$$
which is an isomorphism with respect to both $\leq^*$ and $\leq.$

$($b$)$ If $m+1 <n,$ then there exists
$$i: \Rforce_w / p \rightarrow \Rforce_{w_m}/p^{\leq m} \times \Col(\kappa_{w_m}^{+5}, <\kappa_{w_{m+1}}) \times \Rforce_w / p^{>m+1}$$
which is an isomorphism with respect to both $\leq^*$ and $\leq.$ \hfill$\Box$
\end{lemma}
\begin{lemma}  \label{thm:prikry property}
  $(\RR_w, \leq, \leq^*)$ satisfies the Prikry property, i.e., for any $p \in \Rforce_w$ and any statement
  $\sigma$ in the forcing language of $(\Rforce_w, \leq),$ there exists $q \leq^* p$ which decides $\sigma.$
\end{lemma}
\begin{proof}
We follow the argument given in \cite{hayut-eskew}.
We prove the lemma by induction on $\k_w.$ Thus, assuming  it is true for $\Rforce_u$ with $\k_u < \k_w;$ we prove it for $\Rforce_w.$

Suppose that $p \in \Rforce_w$ and $\sigma$ is a statement in the forcing language of $(\Rforce_w, \leq).$
First, we assume that $\len(p)=1$. So let us write it as $p=(w, \lambda, A, H, h) \in \MPB_w$.

Given $q \in \Rforce_w,$ we can write it as $q=d_q^{\frown} \langle q_{\len(q)}  \rangle$, where $d_q \in V_{\k_w}$ and $q_{\len(q)} \in \MPB_w$.
We set $\stem(q)=d_q$ and call it the stem of $q$.
Let
$\mathbf{L}$
be the set of stems of conditions in $\Rforce_w$ which extend $p$:
\[
\mathbf{L}=\{\stem(q): q \in \Rforce_w \text{~and~} q \leq p                  \}.
\]

Suppose $v \in A$ and $s=\langle q_k: k \leq n  \rangle \in \mathbf{L}$, where $q_k=(w_k, \lambda_k, A_k, H_k, h_k)$ (for $k \leq n$), are such that
for some $A_v, H_v, h_v, A'$, $H'$ and $h'$,
$$q=s^{\frown} \langle v, \lambda, A_v, H_v, h_v \rangle ^{\frown} \langle w, \k_v, A', H', h'    \rangle \in \Rforce_w$$
 and
 $$q \leq \Add(p, v).$$
  Then  $\k_{q_n}=\k_{w^q_{n}}=\lambda,$ in particular there are less than $\lambda^{+5}$-many such stems $s$.
 For each $v \in A$ and each stem $s \in \mathbf{L}$ define the sets $D^{\text{top}}(0, s, v)$ and $D^{\text{top}}(1, s, v)$,
  as follows:
\begin{itemize}
\item
$D^{\text{top}}(0, s, v)$ is the set of all $g \leq H(\k_v)$ for which there exist  $A_v, H_v, h_v, A'$ and $H'$ such that
 $s^{\frown} \langle v, \lambda, A_v, H_v, h_v \rangle ^{\frown} \langle w, \k_v, A', H', g    \rangle\leq \Add(p, v)$ and
it decides $\sigma.$

\item $D^{\text{top}}(1, s, v)$ is the set of all $g \leq H(\k_v)$ such that for all $A_v, H_v, h_v, A', H'$ and $g' \leq g,$
$s^{\frown} \langle v, \lambda, A_v, H_v, h_v \rangle ^{\frown} \langle w, \k_v, A', H', g'    \rangle$
does not decide $\sigma.$

\end{itemize}
Clearly, $D^{\text{top}}(0, s, v) \cup D^{\text{top}}(1, s, v)$ is dense in $\Col(\k_v^{+5}, < \k_w)/H(\k_v)$, and so by the distributivity of
$\Col(\k_v^{+5}, < \k_w),$ the intersection
\[
D^{\text{top}}_v= \bigcap_{s \in \mathbf{L} \cap V_{\k_v}} (D^{\text{top}}(0, s, v) \cup D^{\text{top}}(1, s, v))
\]
is also dense in $\Col(\k_v^{+5}, < \k_w)/H(\k_v)$.
Take $\tilde H \in F^*_w$ such that
\[
\tilde A=\{ v \in A: \tilde H(\k_v) \in D^{\text{top}}_v            \} \in \mathcal{F}_w.
\]
Let $H^* \in F^*_w$ extends both of $H$ and $\tilde H.$

Next define the sets
 $D^{\text{low}}(0, s, v)$ and $D^{\text{low}}(0, s, v)$
 as follows:
\begin{itemize}
\item $D^{\text{low}}(0, s, v)$ is the set of all $g \leq h$ for which there exist  $A_v, H_v, A'$ and $H'$ such that
 $s^{\frown} \langle v, \lambda, A_v, H_v, g \rangle ^{\frown} \langle w, \k_v, A', H', H^*(\k_v)    \rangle\leq \Add(p, v)$ and
it decides $\sigma.$

\item $D^{\text{low}}(1, s, v)$ is the set of all $g \leq h$ such that for all $A_v, H_v,  A', H'$ and $g' \leq H^*(\k_v),$
$s^{\frown} \langle v, \lambda, A_v, H_v, g \rangle ^{\frown} \langle w, \k_v, A', H', H^*(\k_v)    \rangle$
does not decide $\sigma.$
\end{itemize}
The set $D^{\text{low}}(0, s, v) \cup D^{\text{low}}(1, s, v)$ is dense in $\Col(\l^{+5}, < \k_v)/h$, and so by the distributivity of
$\Col(\l^{+5}, < \k_v)$, the intersection
\[
D^{\text{low}}_v= \bigcap_{s \in \mathbf{L} \cap V_{\k_v}} D^{\text{low}}(0, s, v) \cup D^{\text{low}}(1, s, v)
\]
is also dense in $\Col(\l^{+5}, < \k_v)/h$.
Take $ \tilde h_v \in D^{\text{low}}_v.$

Now consider $$p'=(w, \lambda, \tilde A, H^*, h)\leq p.$$
For any stem
$s$ of a condition in $\Rforce_w$ extending $p'$ and every $\alpha < \len(w),$
let $A(s, \alpha) \in \mathcal{F}_w$ be such that one of the following three possibilities holds for it:
\begin{enumerate}
\item [($1_{s, \alpha}$):] For every $v \in A(s, \alpha)$ there exists $q' \leq p'$ such that  $q'$ forces $\sigma$
and $q'$ is of the form
$$q'= s^{\frown} \langle   v, \lambda, A'_{s, v}, H'_{s, v}, h'_{s, v}      \rangle ^{\frown} \langle w, \kappa_v, A_{s, v},  H_{s, v}, h_{s,v}     \rangle,$$ for some $A'_{s, v}, H'_{s, v}, h'_{s, v} \leq \tilde h_v, A_{s, v}, H_{s, v}$ and $h_{s, v} \leq H^*(\k_v).$

\item [($2_{s, \alpha}$):]  For every $v \in A(s, \alpha)$ there exists $q' \leq p'$ such that  $q'$ forces $\neg \sigma$
and $q'$ is of the above form.

\item [($3_{s, \alpha}$):] For every $v \in A(s, \alpha)$ there does not exist $q' \leq p'$ of the  above form  such that $q'$ decides  $\sigma.$
\end{enumerate}
For every $v$,  we may suppose that $H_{s, v}, h_{s, v}, H'_{s, v}$ and $h'_{s, v}$'s depend only on $v$, and so we denote them by $H_v, h_v, H'_v$ and $h'_v$ respectively. For each $\alpha$
let $A(\alpha)=  \bigtriangleup_{s} A(s, \alpha)$ be the diagonal intersection of the $A(s, \alpha)$'s and set
\[
A^*= A' \cap \bigcup_{\alpha < \len(w)} A(\alpha) \in \mathcal{F}_w.
\]
Also let $$p^*= (w, \lambda, A^*, H^*, h).$$

Note that if $v \in A' \cap A(s, \alpha)$
and if one of the   $(1_{s, \alpha})$
or $(2_{s, \alpha})$ happen, then we may take $h_v=H^*(\k_v)$ and $h'_v=\tilde h_v$. This is because if one of these possibilities happen, then
$H^*(\k_v) \in D(0, s, v)$, so there are
 $\tilde A_v, \tilde H_v,  \tilde A'$ and $\tilde H'$ such that
 $$q=s^{\frown} \langle v, \lambda, \tilde A_v, \tilde H_v, \tilde h_v \rangle ^{\frown} \langle w, \k_v, \tilde A', \tilde H', H^*(\k_v)    \rangle\leq \Add(p, v)$$
  and
it decides $\sigma.$
On the other hand, there exists
\[
q'= s^{\frown} \langle   v, \lambda, A'_{s, v}, H'_{s}, h'_{s}      \rangle ^{\frown} \langle w, \kappa_v, A_{s, v},  H_{v}, h_{v}     \rangle \leq p'
\]
which also decides $\sigma.$
But the conditions $q$
and $q'$ are compatible and they decide the same truth value; hence we can take $h_v=H^*(\k_v)$ and $h'_v=\tilde h_v$.

We show that there exists a direct extension of $p^*$ which decides $\sigma.$
Assume not and let $r \leq p^*$ be of minimal length which decides $\sigma,$ say it forces $\sigma$.  Let us write
\[
\stem(r)=s^{\frown} \langle u, \lambda, A^r, H^r, h^r     \rangle,
\]
where $s \in V_{\k_u}.$
By our assumption, there exists $\alpha < \len(w)$
such that $A(s, \alpha) \in w(\alpha)$ satisfies
$(1_{s, \alpha})$,  so for every $v \in A(s, \alpha)$,
there exists
$q'_v \leq p'$ such that  $q'_v$ forces $\sigma$
and $q'_v$ is of the form
$$q'_v= s^{\frown} \langle   v, \lambda, A'_{v}, H'_{v}, \tilde h_{v}      \rangle ^{\frown} \langle w, \kappa_v, A_{v},  H_{v}, H^*(\k_v)     \rangle,$$ for some $A'_{v}, H'_{v}, , A_{v}$ and $H_{v}.$

We show that there exists $q^* \leq p^*$ with $\stem(q^*)=s$
such that every extension of $q^*$  is compatible with  $q'_v,$ for some $v \in A(s, \alpha).$
This property implies
that $q^*$ forces $\sigma$, contradicting the minimal choice of $\len(r)$. We note that by the
definition of extension in the forcing $\Rforce_w$, we may assume from this point on that $s$ is empty.

Consider the map $\phi: A(\langle \rangle, \alpha) \to V$ which is defined by
$$\phi: v \mapsto (\phi_0(v), \phi_1(v))=(A_\nu, H_\nu).$$
As $A(\langle \rangle, \alpha) \in w(\alpha)$, we have $w \upharpoonright \alpha \in j(A(\langle \rangle, \alpha))$ (where $j$ is the constructing embedding for $w$).
Let
\[
(A^{<\alpha}, H^{<\alpha}) = j(\phi)(w \upharpoonright \alpha).
\]
Also let
\[
A^\alpha=\{v \in A(\langle \rangle, \alpha): A^{<\alpha} \cap V_{\k_v}=A_v \text{~and~} H^{<\alpha} \upharpoonright V_{\k_v}=H_v                   \}
\]
and
\[
A^{>\alpha} = \bigcup_{\alpha < \beta < \len(w)}\{v \in A^*:  A^\alpha \cap V_{\k_v} \in w(\beta)               \}.
\]
Then $A^{**}=A^{<\alpha} \cup A^\alpha \cup A^{>\alpha} \in \mathcal{F}_w$.
Set $H^{**}=H^{<\alpha} \wedge H^*$ and finally set
\[
q^*= \langle  w, \lambda, A^{**}, H^{**}, h      \rangle \leq p^*.
\]
We show that $q^*$ is as required. Thus let
$$q= \langle (w_k, \lambda_k, A_k, H_k, h_k): k \leq n                  \rangle$$
be an extension of $q^*.$ There are various cases:
\begin{enumerate}
\item There is no index $k$ such that $\len(u_k)>0$ and $(A^\alpha \cup A^{>\alpha}) \cap V_{\k_{w_k}} \in \bigcup_{\beta < \len(w_k)} w_k(\beta).$
Then pick some non-trivial measure sequence $v \in A^\alpha \cap A_n,$ and note that for all $k<n,$
$A^{<\alpha} \cap A_k \in \bigcap_{\beta < \len(w_k)}w_k(\beta)$. Then one can easily show that $q$ is compatible with $q'_v$.

\item There is an index $k$ with $\len(u_k)>0$ and $(A^\alpha \cup A^{>\alpha}) \cap V_{\k_{w_k}} \in \bigcup_{\beta < \len(w_k)} w_k(\beta)$
and $A^\alpha \in w_k(\beta)$ for some $\beta < \len(w_k).$
Let us pick $k$ to be the least such an index.
Let $v \in A_k$ be such that $A^{<\alpha} \cap A_k \in \bigcap_{\beta < \len(v)} v(\beta)$.
Then $q$ is compatible with $q'_v$.

\item There is an index $k$ with $\len(u_k)>0$ and $(A^\alpha \cup A^{>\alpha}) \cap V_{\k_{w_k}} \in \bigcup_{\beta < \len(w_k)} w_k(\beta)$
and $A^{>\alpha} \in w_k(\beta)$ for some $\beta < \len(w_k).$
Then by our choice of $A^{>\alpha},$ there is some $v \in A_k$
that can be added to $q$ such that we reduce to the  case (2).
\end{enumerate}
This completes the proof for the case $\len(p)=1.$ We now prove the lemma for an arbitrary condition $p$, by induction on $\len(p)$.
Thus suppose that $\len(p) \geq 2$; say
$$p= s^{\frown} \langle (u, \lambda', A', H', h') \rangle ^{\frown} \langle  (w, \lambda, A', H', h)          \rangle.$$
By the factorization Lemma \ref{thm:factorization lemma}, we have
\[
\Rforce_w/p \simeq (\Rforce_u / s^{\frown} \langle (u, \lambda', A', H', h') \rangle) \times (\Rforce_w / \langle  (w, \lambda, A', H', h)          \rangle).
\]
Let $\langle  s_i: i < \k_u     \rangle$
enumerate $\mathbf{L} \cap V_{\k_u}$, and define by recursion on $i$ a $\leq^*$-decreasing chain $\langle p_i: i \leq \k_u    \rangle$
of conditions in $\Rforce_w / \langle  (w, \lambda, A', H', h)          \rangle$
as follows:

Set $p_o=\langle  (w, \lambda, A', H', h)          \rangle.$ Given $p_i,$ let $p_{i+1} \leq^* p_i$
decide whether there is a condition in $\Rforce_u / s^{\frown} \langle (u, \lambda', A', H', h') \rangle$
with stem $s_i$ which decides $\sigma$ and if so, then it forces one of $\sigma$
or $\neg \sigma.$
At limit ordinals $i \leq \k_u,$ use the fact that $(\Rforce_w / \langle  (w, \lambda, A', H', h)          \rangle, \leq^*)$
is $\k_w$-closed to find
a $p_i$ which $\leq^*$-extends all
$p_j, j<i.$

By our construction,
\begin{center}
$\Vdash_{\Rforce_u / s^{\frown} \langle (u, \lambda', A', H', h') \rangle}$``$p_{\k_u}$ decides $\sigma$''.
\end{center}
By the induction hypothesis, there exists $q \leq^* s^{\frown} \langle (u, \lambda', A', H', h') \rangle$ which decides which way
$p_{\k_u}$ decides $\sigma,$ and then $q^{\frown} p_{\k_u} \leq^* p$ decides $\sigma.$

The lemma follows.
\end{proof}
Now suppose that $w=u \upharpoonright \kappa^+$, where $u$ is the measure sequence constructed by the pair $(j, F)$ and let $K \subseteq \Rforce_w$ be a generic filter over $V$. Set
\begin{center}
$C=\{\kappa_u \mid \exists p \in K, \exists \xi < \len(p), p_\xi= (u, \lambda, A, H, h)  \}.$
\end{center}
By standard arguments, $C$ is a club of $\kappa$, also we can suppose that $\min(C)=\aleph_0$.
Let $\langle \kappa_\xi: \xi<\kappa \rangle$ be the increasing enumeration of the club $C$ and let $\vec{u}=\langle u_\xi \mid \xi <\kappa \rangle$
be the enumeration of
$$\{u \mid \exists p \in K, \exists \xi< \len(p), p_\xi= (u, \lambda, A,H, h) \}$$
 such that for $\xi<\zeta<\kappa, \kappa_{u_\xi}=\kappa_\xi < \kappa_\zeta=\kappa_{u_\zeta}.$ Also let $\vec{F}=\langle F_\xi \mid \xi<\kappa \rangle$ be such that each $F_\xi$ is $\Col(\kappa_\xi^{+5}, <\kappa_{\xi+1})$-generic over $V$ produced by $K.$
 \begin{lemma}\label{thm: bounding sets}
 \begin{enumerate}
\item [(a)] $V[K]=V[\vec{u}, \vec{F}].$

\item [(b)] For every limit ordinal $\xi<\kappa, \langle \vec{u}\upharpoonright \xi, \vec{F}\upharpoonright \xi  \rangle$  is $\RR_{u_\xi}$-generic over $V$,
and $\langle \vec{u} \upharpoonright [\xi, \kappa), \vec{F}\upharpoonright$
 $[\xi, \kappa) \rangle$ is $\RR_w$-generic over $V[\vec{u}\upharpoonright \xi, \vec{F}\upharpoonright \xi].$

\item [(c)] For every $\gamma <\kappa$ and every $A \subseteq \gamma$ with $A\in V[\vec{u}, \vec{F}],$ we have $A\in V[\vec{u}\upharpoonright \xi, \vec{F}\upharpoonright \xi],$
 where $\xi$ is the least ordinal such that $\gamma< \kappa_\xi.$
\end{enumerate}
\end{lemma}
\begin{proof}

$($a$)$ It suffices to show that $K$ is definable from $\vec{u}$ and $ \vec{F}$. Let $K'$ be the set of all conditions $p \in \Rforce_w$
such that
\begin{itemize}
\item For all measure sequences $u \in V_\kappa$, if $u$ appears in $p$, then $u=u_\xi$, for some $\xi < \kappa$,

\item For all $\xi< \kappa$, there exists $q \leq p$ such that $u_\xi$ appears in $q$,

\item If $f \in V_\kappa$ appears in $p$, then $f \subset F_\xi$, for some $\xi< \kappa$,

\item For all $\xi< \kappa$ and all $f \in \ps(F_\xi)\cap\Col(\kappa_{\xi}^{+5},< \kappa_{\xi+1})$, there exists $q \leq p$ such that $f$ appears in $q$.
\end{itemize}
It is clear that $K' \in V[\vec{u}, \vec{F}]$. It is also easily seen that $K'$ is a filter which includes $K$. It follows from the genericity of $K$
that $K=K'$. So $K \in V[\vec{u}, \vec{F}]$, as required.

$($b$)$ Follows from $($a$)$ and the factorization lemma~\ref{thm:factorization lemma}.

$(c)$ First note that $\nu$ is not a limit ordinal, so assume $\nu=\xi+1$ is a successor ordinal (if $\nu=0$, then the proof is similar). Let $p \in K$ be such that $p$ mentions both $u_{\xi}$  and $u_{\xi+1}$, say
$u_{\xi}=u^{p_m}$ and $u_{\xi+1}=u^{p_{m+1}}$.
 By the Factorization Lemma~\ref{thm:factorization lemma},
 \[
 \Rforce_w / p \simeq \Rforce_{u_\xi}/p^{\leq m} \times \Col(\kappa_\xi^{+5}, <\kappa_{\xi+1}) \times \Rforce_w / p^{>m+1}.
 \]
Let $\dot{A}$ be an $\Rforce_w$-name for $A$
such that $\Vdash_{\Rforce_w}\dot{A} \subseteq \gamma$. Let $\dot{B}$ be an $\Rforce_w / p^{> m+1}$-name for a subset
of $\Rforce_{u_\xi}/p^{\leq m} \times \Col(\kappa_\xi^{+5}, <\kappa_{\xi+1}) \times \gamma$ such that
\[
\Vdash_{\Rforce_w / p^{>m+1}} \forall \eta < \gamma, ((r,f, \eta) \in
\dot{B}
\iff (r, f) \Vdash_{\Rforce_{u_\xi}/p^{\leq m} \times \Col(\kappa_\xi^{+5}, <\kappa_{\xi+1})} \eta \in \dot{A} ).
\]
Let $\langle y_\alpha: \alpha < \kappa_{\xi+1}      \rangle$ be an enumeration of $\Rforce_{u_\xi}/p^{\leq m} \times \Col(\kappa_\xi^{+5}, \kappa_{\xi+1}) \times \gamma$.
Define a $\leq^*$-decreasing sequence $\langle  q_\alpha \mid \alpha < \kappa_{\xi+1}        \rangle$
of conditions in $\Rforce_w / p^{>m+1}$ such that for all $\alpha, q_\alpha$ decides ``$y_\alpha \in \dot{B}$''. This is possible as $(\Rforce_w / p^{> m+1}, \leq^*)$
is $\kappa_{\xi+1}^+$-closed and by Lemma~\ref{thm:prikry property} it satisfies the Prikry property. Let $q \leq^* q_\alpha$ for all $\alpha < \kappa_{\xi+1}$.
Then $q$ decides each ``$y_\alpha \in \dot{B}$''. It follows that
 $A\in V[\vec{u}\upharpoonright \nu, \vec{F}\upharpoonright \nu]$
\end{proof}

We now state a geometric characterization of generic filters for $\MRB_{w}$. Such a characterization was first given by Mitchell \cite{mitchell82} for Radin forcing. The characterization given bellow is essentially due to Cummings \cite{cummings}.
\begin{lemma} (Geometric characterization)
\label{geometric characterization}
The pair $(\vec{u}, \vec{F})$ is $\MRB_{w}$-generic over $V$ if and only if it satisfies the following conditions:
\begin{enumerate}
\item If $\xi < \k$ and $\len(u_\xi)>1,$ then the pair $(\vec{u} \upharpoonright \xi, \vec{F} \upharpoonright \xi)$ is $\MRB_{u_\xi}$-generic over $V$.

\item For all $A \in V_{\kappa+1}~(A \in \mathcal{F}_w \iff \exists \a<\kappa~ \forall \xi>\a,~ u_\xi \in A)$.

\item For all $f \in w(1)$ there exists $\a < \kappa$ such that $ \forall \xi>\a, f(\k_\xi) \in F_\xi.$
\end{enumerate}
\end{lemma}
\

As $\len(w)=\kappa^+,$ it follows from Mitchell \cite{mitchell82} (see also \cite{gitik}) that
\begin{lemma}  \label{thm:preserving inaccessibility}
$\kappa$ remains strongly inaccessible in $V[K]$.
\end{lemma}
\begin{proof}
We follow Cummings \cite{cummings}.
Suppose not and let $p \in \Rforce_w, \delta < \kappa$ and $\dot{f}$ be such that
\begin{center}
$p \Vdash$``$\dot{f}: \delta \to \k$ is cofinal''.
\end{center}
Let $\theta > \kappa$ be large enough regular such that $p, \dot{f}, w, \Rforce_w \in H(\theta)$
and let $\mathbf{X} \prec H(\theta)$
be such that
\begin{enumerate}
\item $p, \k^+, \dot{f}, w, \Rforce_w \in \mathbf{X}.$
\item $V_{\k} \subseteq \mathbf{X}.$
\item $^{<\k}$$\mathbf{X} \subseteq \mathbf{X}.$
\item $|\mathbf{X}|=\k$.
\end{enumerate}
Let $\pi: \mathbf{X} \to \mathbf{N}$ be the Mostowski collapse of $\mathbf{X}$ onto a transitive model $\mathbf{N}.$ Note that
$\pi \upharpoonright \mathbf{X} \cap V_{\k+1} =id \upharpoonright \mathbf{X} \cap V_{\k+1}.$

Let $v=\pi(w)$ and $\beta=\pi(\k^+).$ Then
\[
\pi(\mathcal{F}_w) = \mathcal{F}_w \cap \mathbf{X} = \mathcal{F}_w \cap \mathbf{N}
\]
and
\[
\forall \alpha  \in \mathbf{X} \cap \k^+,~  \pi(v(\alpha))=v(\pi(\alpha)) = v(\alpha) \cap \mathbf{X} = w(\alpha  ) \cap \mathbf{N}.
\]
Let $\bar \beta = \sup (\mathbf{X} \cap \k^+) < \k^+.$ Using  Lemma \ref{geometric characterization},
if $G$ is $\Rforce_{w \upharpoonright  \gamma}$-generic over $V$, where $\bar \beta \leq \gamma < \k^+$, then $G$ is $\pi(\Rforce_w)$-generic over $N$.
But note that for any limit ordinal $\gamma$ as above with $\cf(\gamma)< \k,$ we have
\[
 \Vdash_{\Rforce_{w \upharpoonright  \gamma}} \text{``}\cf(\k)=\cf(\gamma)\text{''}.
\]
We get a contradiction and the lemma follows.
\end{proof}
It follows that
\[
CARD^{V[K]} \cap \kappa = \bigcup_{\a \in C} \{\alpha, \alpha^+, \alpha^{++}, \a^{+3}, \a^{+4}, \alpha^{+5}  \}.
\]
As every limit point of $C$ is singular in $V[K],$ it follows that $\kappa$ is the least inaccessible cardinal. Also note that $\lim(C)$, the set of limit points of $C$, is exactly the set of all singular cardinals below $\k$ in $V[K]$.



\subsection{The final model}
\label{sec: final model}
Suppose that  $GCH$ holds and $\k$ is an $H(\l^{++})$-hypermeasurable cardinal where  $\l$ is a measurable cardinal above $\k$.
We define a generic extension $W$ of $V$ which satisfies $W \models$``$\k$ is inaccessible and for all singular cardinals $\delta< \k, 2^\delta=\delta^{+3}$
and $\TP(\delta^{++})$ holds''.

We will next give a vague and incomplete description of the way the model $W$ is constructed.
Thus we start with $GCH$ and an $H(\l^{++})$-hypermeasurable embedding $j: V \to M$ with $i: V \to N$
being its ultrapower embedding.

We first define a generic extension $V^1$ of $V$
in which $\k$ remains $H(\l^{++})$-hypermeasurable as witnessed by an elementary embedding $j^1: V^1 \to M^1$ which extends $j$  and in which there exists a generic filter for a suitably chosen forcing notion defined in $\Ult(V^1, U^1),$ where $U^1$ is the normal measure on $\k$ derived from $j^1$.

We then define a generic extension $V^2$ and $V^1$ in which $\kappa$ remains $H(\k^{++})$-hypermeasurable  witnessed by an elementary embedding $j^2: V^2 \to M^2$ which extends $j^1$ and such that if $U^2$ is the normal measure derived from $j^2,$ then for $U^2$-measure one many $\delta < \k$ we have $\delta$ is measurable, $2^\delta=\delta^{+3}$ and $\TP(\delta^{++})$ holds. Further the model $V^2$ satisfies the hypotheses at the beginning of Subsection \ref{measure sequences}.

Working in $V^2$ we force with the forcing notion $\Rforce_w,$ for a suitably chosen measure sequence $w$, to find a generic extension $V^3$
 of $V^2$. We  show that in $V^3$ the tree property holds at double successors of the limit points of the Radin club and using it we conclude that $W=V^3$ is the required model.

Thus suppose that $V$ satisfies $GCH$ and let $\k$ be an $H(\l^{++})$-hypermeasurable cardinal in it where  $\l$ is the least measurable cardinal above $\k$.
 Also let $f: \k \to \k$ be defined by
\[
f(\a)=(\min\{\beta> \a: \beta \text{~is a measurable cardinal~}     \})^+.
\]
 Let $j: V \to M$ witness the  $H(\l^{++})$-hypermeasurability of $\k$ and suppose $j$  is generated by a $(\k, \l^{++})$-extender, i.e.,
\[
M=\{j(g)(\a): g: \k \to V, \a < \l^{++}             \}.
\]
Then $j(f)(\k)=\l^+$.
Also let $U$ be the normal measure derived from $j$; $U=\{ X \subseteq \k: \k \in j(X)    \}$ and let $i: V \to N \simeq \Ult(V, U)$
be the induced ultrapower embedding. Let $k: N \to M$ be elementary so that $j=k \circ i.$
\begin{notation}
\begin{enumerate}
\item [(a)] For each infinite cardinal $\alpha \leq \k$ let $\alpha_*$ denote the least measurable cardinal above $\alpha.$ Note that $\k_*=\lambda.$

\item [(b)] For an infinite cardinal $\a \leq \k$ let $\mathbb{M}_\a= \mathbb{M}(\alpha, \alpha_*, \alpha_*^+)$.
\end{enumerate}
\end{notation}
We start with the following lemma.
\begin{lemma}
\label{a very basic iteration lemma}
 There exists a cofinality preserving generic extension $V^1$ of $V$ satisfying the following conditions:
\begin{enumerate}
\item [(a)] There is $j^1: V^1 \to M^1$ with critical point $\k$
such that $H(\l^{++}) \subseteq M^1$ and $j^1 \upharpoonright V= j.$

\item [(b)] $j^1$ is generated by a $(\k, \l^{++})$-extender.

\item [(c)] If $U^1$ is the normal measure derived from $j^1$ and if $i^1: V^1 \to N^1 \simeq \Ult(V^1, U^1)$
is the ultrapower embedding, then there exists  $\bar g \in V^1$
which is $i^1(\Add(\k, \l^{+})_{V^1})$-generic over $N^1.$
Further $i^1 \upharpoonright V=i.$
\end{enumerate}
\end{lemma}
\begin{proof}
See \cite{friedman-honzik2} Theorem 3.1 and Remark 3.6.
\end{proof}
 Let $V^1$ be the model constructed above.
\begin{lemma}
\label{preparation model for double succseeor model}
Work in $V^1$.
There exists a forcing iteration $\MPB_\k$ of length $\k$ such that if $G_\k \ast g$ is $\MPB_\k \ast \dot{\mathbb{M}}(\k, \l, \l^+)$-generic over
$V^1$ and $V^2=V^1[G_\k \ast g]$, then the following conditions hold:
\begin{enumerate}
\item [(a)] There is $j^2: V^2 \to M^2$ with critical point $\k$
such that $H(\k^{++}) \subseteq M^2$ and $j^2 \upharpoonright V^1= j^1.$

\item [(b)] $j^2$ is generated by a $(\k, \l^{+})$-extender.

\item [(c)] $V^2 \models$``$\l= \k^{++}+2^\k=\l^+=\k^{+3} + \TP(\l)$''.

\item [(d)] If $U^2$ is the normal measure derived from $j^2$ and if $i^2: V^2 \to N^2 \simeq \Ult(V^2, U^2)$
is the ultrapower embedding, then there exists  $F \in V^2$
which is $\Col(\k^{+5}, < i(\k))_{N^2}$-generic over $N^2.$

\end{enumerate}
\end{lemma}
\begin{proof}
Work in $V^1$. Factor $j^1$ in two steps through the models
\[
N'=\text{~the transitive collapse of~}\{j^1(f)(\k): f: \k \to V^1            \}
\]
\[
 \bar N'=\text{~the transitive collapse of~}\{j^1(f)(\a): f: \k \to V^1, \a < \l^+            \}.
\]
$N'$ is the familiar ultrapower approximating $M^1$, while $\bar N'$ corresponds to the
extender of length $\l^+$. We have maps

$\hspace{6.cm}$$i': V^1 \to N'$,

$\hspace{6.cm}$$k': N' \to M^1,$

$\hspace{6.cm}$$\bar i': N' \to \bar N',$

$\hspace{6.cm}$$\bar k': \bar N' \to M^1$

such that

$\hspace{5.cm}$$k' \circ i'=j^1$~~~~ $\&$~~~~ $\bar k' \circ \bar i' =k'$.

Let
\[
\MPB_{\k}=\langle  \langle \MPB_\a: \a \leq \k      \rangle, \langle \dot{\MQB}_\a: \a < \k          \rangle\rangle
\]
be the reverse Easton iteration of forcing notions  such that
\begin{enumerate}
\item If $\a < \k$ is a measurable limit of measurable cardinals, then $\Vdash_\a$``$\dot{\MQB}_\a =\dot{\mathbb{M}}(\a, \a_*, \a_*^+)$''.
\item Otherwise, $\Vdash_\a$``$\dot{\MQB}_\a$ is the trivial forcing''.
\end{enumerate}
Let
$G_k \ast g$
be $\MPB_\k \ast \dot{\mathbb{M}}(\k, \l, \l^+)$-generic over $V^1$.

Note that  we can factor $\mathbb{M}(\k, \l, \l^+)$
as
$$\mathbb{M}(\k, \l, \l^+)=\Add(\k, \l^+) \ast \dot{\tilde \MQB},$$
where $\dot{\tilde \MQB}$ is forced to be $\k^+$-distributive. Let us factor $g$ as  $g=g(0) \ast g(1).$

As $\MPB_\k$ is computed in all of the models the same and the embeddings $\bar i', k', \bar k'$ have critical point bigger than $\k,$ we can easily lift them to get

$\hspace{6.cm}$$k': N'[G_\k] \to M^1[G_\k],$

$\hspace{6.cm}$$\bar i': N'[G_\k] \to \bar N'[G_\k],$

$\hspace{6.cm}$$\bar k': \bar N'[G_\k] \to M^1[G_\k]$,

where $\bar k' \circ \bar i' =k'.$ The models $\bar N'[G_\k]$ and $M^1[G_\k]$ are closed under $\k$-sequences in $V^1[G_\k]$
and they compute the cardinals up to $\l^+$ in the correct way, in particular, the least measurable above $\k$
in these models is $\l$, and so if we set $\MQB_\k =\Add(\k, \l^+)_{V^1[G_\k]},$ then it is computed in the same way in the models
$\bar N'[G_\k]$ and $M^1[G_\k]$, i.e.,
$$\MQB_\k=(\MQB_\k)_{\bar N'[G_\k]}=(\MQB_\k)_{M^1[G_\k]}.$$
On the other hand
\[
(\MQB_\k)_{N'[G_\k]}=\Add(\k, \bar \l)_{V^1[G_\k]},
\]
where  $\bar \l = (i'(f)(\k)^+)_{N'}.$ Note that $\k^+ < \bar \l < \k^{++}.$
Factor $g(0)$ as $g(0)_1 \times g(0)_2$, which corresponds to
\[
\Add(\k, \l^+)_{V^1[G_\k]} = \Add(\k, \bar \l)_{V^1[G_\k]} \times \Add(\k, \l^+ \setminus \bar \l)_{V^1[G_\k]}.
\]
We build further
extensions

$\hspace{5.cm}$$k': N'[G_\k][g(0)_1] \to M^1[G_\k][g(0)],$

$\hspace{5.cm}$$\bar i': N'[G_\k][g(0)_1] \to \bar N'[G_\k][g(0)],$

$\hspace{5.cm}$$\bar k': \bar N'[G_\k][g(0)] \to M^1[G_\k][g(0)]$,

still preserving the relation $\bar k' \circ \bar i' =k'.$ 

Let us now write $i'(\MPB_\k)$ as
$$i'(\MPB_\k)=\MPB_\k \ast \dot{\Add}(\k, \bar \l)\ast \dot{\MQB}' \ast i'(\MPB_\k)_{(\k+1, i'(\k))} \ast \dot{\Add}(i'(\k), i'(\bar\l))\ast i'(\dot{\MQB}'),$$
where $\Vdash_{\MPB_\k \ast \dot{\Add}(\k, \bar \l)}$``$ \dot{\MQB}'$ is $\k^+$-distributive''.
Let $g(1)'$ be the filter generated by $i'^{''}(g(1))$. Then $g(1)'$ is $\dot{\MQB}'[G_\k \ast g(0)_1]$-generic over 
$N'[G_\k \ast g(0)_1].$

By standard arguments, we can find $H \in V^1[G_\k][g(0)_1]$,
which is $i'(\MPB_{\k})_{(\k+1, i'(\k))}$-generic over $N'[G_\k][g(0)_1]$ and hence we can get 

$\hspace{5.cm}$$i': V^1[G_\k] \to N'[i'(G_\k)].$

Also transfer $H$ along $\bar i', k'$ to get


$\hspace{5.cm}$$k': N'[i'(G_\k)] \to M^1[j^1(G_\k)],$

$\hspace{5.cm}$$\bar i': N'[i'(G_\k)] \to \bar N'[ \bar i' \circ i'(G_\k)],$

$\hspace{5.cm}$$\bar k': \bar N'[\bar i' \circ i'(G_\k)] \to M^1[j^1(G_\k)]$,

where all the maps are defined in $V^1[G_\k][g(0)]$. Let $l =\bar i' \circ i'.$
Since $\MPB_\k$ has size $\k$ and is $\k$-c.c., so the term forcing
\[
\Add(\k, \l^+)_{V^1[G_\k]} / \MPB_\k
\]
is forcing isomorphic to $\Add(\k, \l^+)_{V^1}$ (see \cite{cummings} Fact 2, $\S$1.2.6). By our assumption, we have
$\bar g \in V^1$,
which is $i'(\Add(\k, \l^+)_{V^1})$-generic over $N'$,
and using it we can define $g_a$ which is
\[
\Add(i(\k), i(\l^+))_{N'[i'(G_\k)]}
\]
generic over $N'[i'(G_\k)]$. Using the fact that
\[
V^1[G_\k][g(0)_1] \models ~``~^\k N'[i'(G_\k)] \subseteq N'[i'(G_\k)] \text{''}
\]
we also build $F$, which is $\Col(\k^{+5}, <i'(\k))_{N'[i'(G_\k)]}$ generic over $N'[i'(G_\k)]$.
Note that $g_a$ and $F$ are mutually generic.

Transfer $g_a$ and $F$ along $\bar i'$ to get new generics $\bar g_a$ and $\bar F.$
Now using Woodin's surgery argument, we can alter the filter $\bar g_a$ to find a generic filter $h_a$
with the additional property
$l''[g(0)] \subseteq h_a$. Also $h_a$ is easily seen to be mutually generic with
$\bar F.$

We now transfer $h_a$ along $\bar k'$ to get $H_a$ which is  $j^1(\MQB_\k)$-generic over $M^1$. Further,
$j^{1''}[g(0)] \subseteq H_a$, so we can build maps

$\hspace{4.cm}$$\bar j: V^1[G_\k \ast g(0)] \to M^1[j^1(G_\k \ast g(0))],$

$\hspace{4.cm}$$\bar k': \bar N'[l(G_\k \ast g(0))] \to M^1[j^1(G_\k \ast g(0))]$,

$\hspace{4.cm}$$l: V^1[G_\k \ast g(0)] \to \bar N'[ l(G_\k \ast g(0))].$

such that $\bar j= \bar k' \circ l.$ 

Now let us look at $\dot{\tilde\MQB}[G_\k \ast g(0)]$. It is $\k^+$-distributive in $V^1[G_\k \ast g(0)]$,
so  we can further extend the above embeddings and get

$\hspace{3.cm}$$\bar j: V^1[G_\k \ast g(0) \ast g(1)] \to M^1[j^1(G_\k \ast g(0) \ast g(1))],$

$\hspace{3.cm}$$\bar k': \bar N'[l(G_\k \ast g(0)\ast g(1))] \to M^1[j^1(G_\k \ast g(0)\ast g(1))]$,

$\hspace{3.cm}$$l: V^1[G_\k \ast g(0)\ast g(1)] \to \bar N'[ l(G_\k \ast g(0)\ast g(1))].$

Let
$$V^2=V^1[G_\k \ast g(0) \ast g(1)],$$
 $$M^2=M^1[j^1(G_\k \ast g(0) \ast g(1))] $$
and $$N^2=\bar N'[ l(G_\k \ast g(0)\ast g(1))].$$
Also let $j^2=\bar j.$ We argue
\[
\Ult(V[G_\k \ast g(0) \ast g(1)], U^2) \simeq N^2,
\]
where $U^2$ is the normal measure derived from $j^2$.
To see this, factor $l$ through $l^\dag: V^2 \to N^\dag \simeq \Ult(V^2, U'),$
where $U'$ is the normal measure derived from $l.$ Also let $k^\dag: N^\dag \to N^2.$
Then
$P(\k)_{V^2} \subseteq N^\dag$ and
 $N^\dag \models$``$2^\k=\l^+$''.
So $\crit(k^\dag)> \l^+$. Since $\l^+ \subseteq \range(k^\dag)$
and $N^2$ is generated by a $(\k, \l^+)$-extender, we have $N^2=N^\dag$
and we are done.

So if we let $i^2=l,$ then
\[
i^2: V^2 \to N^2
\]
is the ultrapower embedding.
Finally note that $F$ is generic for the appropriate collapse ordering.
The lemma follows.
\end{proof}
Note that in the model $V^2=V^1[G_\k \ast g]$, the following conditions are satisfied:
\begin{itemize}
\item  $V^2 \models$``$\l= \k^{++}+2^\k=\l^+=\k^{+3}$''.

\item There is $j^2: V^2 \to M^2$ with critical point $\k$
such that $H(\k^{++}) \subseteq M^2$ and $j^2 \upharpoonright V^1= j^1.$

\item  $j^2$ is generated by a $(\k, \l^+)$-extender.

\item  If $U^2$ is the normal measure derived from $j^2$ and if $i^2: V^2 \to N^2 \simeq \Ult(V^2, U^2)$
is the ultrapower embedding, then there exists  $F \in V^2$
which is $\Col(\k^{+5}, < i(\k))_{N^2}$-generic over $N^2.$
\end{itemize}
Thus the hypotheses of the beginning of Subsection  \ref{measure sequences}
are satisfied, and so, working in $V^2$, we can construct the pair $(j, F)$.
Let $u$ be the measure sequence constructed from it. Set $w= u \upharpoonright \k^+$ and let $\MRB_{w}$
be the corresponding forcing notion as in Definition \ref{radin forcing definition}. Also let $K$ be $\MRB_{w}$-generic over $V^2.$
Build the sequences $\vec{\k}= \langle \kappa_\xi: \xi < \kappa \rangle, \vec{u}= \langle u_\xi: \xi < \kappa \rangle$ and $\vec{F} = \langle F_\xi: \xi<\kappa \rangle$ from $K$, as in Subsection \ref{sec: basic properties}.

\subsection{$\TP(\kappa^{++})$ holds in $V^1[G_\k \ast g \ast K]$}
\label{sec:tree property at double successors}
In this subsection we show that $\TP(\kappa^{++})$ holds in $V^1[G_\k \ast g \ast K],$ and then in the next subsection, we complete the proof of  Theorem \ref{thm:main theorem} by showing that
\begin{center}
$V^1[G_\k \ast g \ast K] \models$``$\TP(\alpha^{++})$ holds for all singular cardinals $\alpha<\kappa$''.
\end{center}
As $V^1[G_\k \ast g \ast K] \models$``$\kappa^{++}=\lambda$'', it suffices to prove the following:
\begin{theorem}
\label{tree property at kappa}
$V^1[G_\k \ast g \ast K] \models$``$\TP(\lambda)$''.
\end{theorem}
The rest of this subsection is devoted to the proof of the above theorem. The proof we present follows ideas of \cite{friedman-honzik2}, but is more involved as instead of working with the Prikry collapse forcing of \cite{friedman-honzik2} we are working with the Radin forcing $\Rforce_w$.

Let
$\dot{\mathbb{M}}$ be  such that
$\Vdash_{\MPB_\k}$``$\dot{\mathbb{M}}=\dot{\mathbb{M}}(\k, \l, \l^+)$''.
\begin{lemma}
\label{chain condition of forcing}
The forcing $\mathbb{P}_\k \ast \dot{\mathbb{M}} \ast \dot{\mathbb{R}}_{w}$
satisfies the $\lambda$-c.c.
\end{lemma}
\begin{proof}
The forcing $\mathbb{P}_\k$ is $\kappa$-c.c. Now the result follows
from the facts that $\MPB_\k$ forces ``$\dot{\mathbb{M}}$ is $\lambda$-c.c.'' (by Lemma \ref{basic facts on mitchell forcing}) and
$\MPB_\k \ast \dot{\mathbb{M}}$ forces ``$\dot{\mathbb{R}}_{w}$ is $\kappa^+$-c.c.''  (by Lemma \ref{chain condition for radin forcing}).
\end{proof}
Assume towards contradiction that $\TP(\lambda)$ fails in
$V^1[G_\k \ast g][\vec{u}, \vec{F}]$
and let $\dot{T} \in V^1[G_\k]$ be an $\mathbb{M} \ast \dot{\mathbb{R}}_{w}$-name
for a $\lambda$-Aronszajn tree in $V^1[G_\k \ast g][\vec{u}, \vec{F}]$.
Suppose for simplicity that the trivial condition forces that $\dot{T}$ is
a $\lambda$-Aronszajn tree and let us view it as a nice name for a subset of $\lambda$; so that $\dot{T}= \bigcup_{\xi<\lambda} \{\check{\xi} \} \times A_\xi,$
where each $A_\xi$ is a maximal antichain in $\mathbb{M} \ast \dot{\mathbb{R}}_{w}$. By Lemma
\ref{chain condition of forcing}, each $A_\xi$ has size less than $\lambda$.

Recall from the remarks after Lemma \ref{basic facts on mitchell forcing} that   the forcing $\mathbb{M}$ is forcing isomorphic to $\Add(\kappa, \lambda^+) \ast \dot{\MQB}$,
where $\dot{\MQB}$ is some  $\Add(\kappa, \lambda^+)$-name for a forcing notion which is forced to be $\kappa^+$-distributive.
\begin{lemma}
\label{mitchell followed by radin}
Work in $V^1[G_\k]$. The set
\[
\{r= ((p, \dot{q}), \check{d}^{\frown}\langle w, \lambda,   \dot{A}, \dot{H}, \check{h} \rangle ) \in \mathbb{M} \ast \dot{\mathbb{R}}_{w}: d, h \in V^1[G_\k] \text{~and~}  \dot{A}, \dot{H} \text{~are~} \Add(\kappa, \lambda^+)-\text{names~}        \}
\]
is dense in $\mathbb{M} \ast \dot{\mathbb{R}}_{w}.$
\end{lemma}
\begin{proof}
Recall that a condition in $\mathbb{R}_{w}$ is of the form $p= d^{\frown} \langle w, \lambda,   A, H, h \rangle$
where
 \begin{enumerate}
 \item $d \in V_\kappa$.

 \item $ \langle w, \lambda,   A, H, h \rangle \in \MPB_w.$

 \item $h \in \Col(\lambda^{+5}, < \kappa).$
 \end{enumerate}
 As $\mathbb{M}$ does not add bounded subsets to $\kappa,$ so any condition $((p, \dot{q}), \dot{d}^{\frown}\langle w, \lambda,   \dot{A}, \dot{H}, \dot{h} \rangle )$ has an extension of the form $((p', \dot{q}'), \check{d}'^{\frown}\langle \kappa, \lambda,   \dot{A}', \dot{H}', \check{h}' \rangle )$,
 where $d', h' \in V^1[G_\k].$

Also note that all conditions in $\MPB_w$ and hence in $\mathbb{R}_{w }$ exist already in the extension by
$\Add(\kappa, \lambda^+)$, the Cohen part of $\mathbb{M}$ (though the definition of $\mathbb{R}_{w}$ may require the whole
$\mathbb{M}$). Thus we can further extend $((p', \dot{q}'), \check{d}'^{\frown}\langle \kappa, \lambda,   \dot{A}', \dot{H}', \check{h}' \rangle )$
to another condition $$((p'', \dot{q}''), \check{d}''^{\frown}\langle w, \lambda,   \dot{A}'', \dot{H}'', \check{h}'' \rangle )$$
where $\dot{A}''$ and $\dot{H}''$ are forced to be $\Add(\kappa, \lambda^+)$-names (over $V^1[G_\k]$). The result follows immediately.
\end{proof}
From now on, we assume that all the conditions in $\mathbb{M} \ast \dot{\mathbb{R}}_{w}$ are of the above form. This is useful in some of the arguments below (see for example Lemma \ref{lem:projection properties}$(a)$).
Let us define
\[
\MCB=\{((p, \emptyset), r):  ((p, \emptyset), r) \in  \mathbb{M} \ast \dot{\mathbb{R}}_{w}           \}
\]
and
\[
 \mathbb{T} = \{(\emptyset, q): (\emptyset, q) \in \mathbb{M}            \}.
\]
Let $\tau: \MCB \times \mathbb{T} \to \mathbb{M} \ast \dot{\mathbb{R}}_{w}$
be defined by
\[
\tau(\langle ((p, \emptyset), r), (\emptyset, q)            \rangle)= ((p, q), r).
\]

\begin{lemma}
\label{lem:projection properties}
\begin{itemize}
\item [(a)] $\tau$ is a projection from  $\MCB \times \mathbb{T}$ onto $\mathbb{M} \ast \dot{\mathbb{R}}_{w}$.

\item [(b)] $\mathbb{T}$ is $\kappa^+$-closed in $V^1[G_\k]$.

\item [(c)] $\mathbb{C}$ is $\kappa^+$-c.c. in $V^1[G_\k]$.
\end{itemize}
\end{lemma}
\begin{proof}

$(a)$ It is clear that $\tau$ is order preserving. Suppose that
$$(p', \dot q'), \dot r') \leq_{\mathbb{M} \ast \dot{\mathbb{R}}_{w}} \tau(\langle ((p, \emptyset), \dot r), (\emptyset, \dot q)            \rangle)= ((p, \dot q), \dot r).$$
We are going to find $p^*, \dot q^*$ and $\dot r^*$ such that  $\langle ((p^*, \emptyset), \dot r^*), (\emptyset, \dot q^*)            \rangle \leq_{\MCB \times \mathbb{T}} \langle ((p, \emptyset), \dot r), (\emptyset, \dot q)            \rangle$ and
$$\tau(\langle ((p^*, \emptyset), \dot r^*), (\emptyset, \dot q^*)            \rangle)= (p^*, \dot q^*), \dot r^*)  \leq_{\mathbb{M} \ast \dot{\mathbb{R}}_{w}} (p', \dot q'), \dot r').$$
Let $p^*=p'$. Let $\dot q^*$ be a name such that
\begin{itemize}
\item $p^* \Vdash$``$\dot q^*=\dot q'$''.
\item If $\tilde{p}$ is incompatible with $p'$, then $\tilde{p}\Vdash$``$\dot q^*=\dot q$''.
\end{itemize}
Also set $\dot r^*=\dot r'$. Then $p^*, \dot q^*$ and $\dot r^*$ are as required.

$(b)$ follows from the fact that $1_{\Add(\kappa, \lambda^+)} \Vdash$``$\Add(\k^+, 1)$ is $\k^+$-closed''.

$(c)$ follows from Lemma \ref{chain condition for radin forcing} and the fact that $\Add(\kappa, \lambda^+)$ is $\k^+$-c.c.
\end{proof}
Let $k: V^1 \to N^1$ witness the measurability of $\lambda$ in $V^1$.
As $|\MPB_\k|=\kappa < \lambda,$  so by the Levy-Solovay's theorem \cite{levy-solovay}, we can lift $k$ to
 $k: V^1[G_\k] \to N^1[G_\k].$

Let $\mathbb{M}^* \ast \dot{\MRB}^*_{w}= k(\mathbb{M} \ast \dot{\MRB}_{w}).$
The next lemma follows from Lemma  \ref{chain condition of forcing}.
\begin{lemma}
\label{regular embedding}
$($in $V^1[G_\k]$$)$
$k \upharpoonright \mathbb{M} \ast \dot{\MRB}_{w}: \mathbb{M} \ast \dot{\MRB}_{w} \to \mathbb{M}^* \ast \dot{\MRB}^*_{w}$ is a regular embedding.
\end{lemma}
\begin{proof}
It is clear that $k$ is order preserving and if $p \perp_{\mathbb{M} \ast \dot{\MRB}_{w}} q$ (~$p$ is incompatible with $q$ in $\mathbb{M} \ast \dot{\MRB}_{w}$), then $k(p)\perp_{\mathbb{M}^* \ast \dot{\MRB}^*_{w}} k(q)$ (~$k(p)$ and $k(q)$
are incompatible in $\mathbb{M}^* \ast \dot{\MRB}^*_{w}).$ Now suppose that $A \subseteq \mathbb{M} \ast \dot{\MRB}_{w}$ is a maximal antichain in
$\mathbb{M} \ast \dot{\MRB}_{w}$. By Lemma  \ref{chain condition of forcing}, $|A|< \lambda$ and so by elementarity of $k$, $k''[A]=k(A)$ is a
maximal antichain in $\mathbb{M}^* \ast \dot{\MRB}^*_{w}.$
\end{proof}
Thus let $g^* \ast K^*$ be $\mathbb{M}^* \ast \dot{\MRB}^*_{w}$-generic over $V^1[G_\k]$
such that $k''[g \ast K] \subseteq g^* \ast K^*.$
It follows that we can lift $k$ to
\[
k: V^1[G_\k \ast g \ast K] \to N^1[G_\k \ast g^* \ast K^*].
\]
Hence, in $V^1[G_\k],$ by Lemma \ref{regular embedding}, there is a projection
\[
\pi: \mathbb{M}^* \ast \dot{\MRB}^*_{w} \to RO( \mathbb{M} \ast \dot{\MRB}_{w}),
\]
 where $RO( \mathbb{M} \ast \dot{\MRB}_{w})$ denotes the Boolean completion of $ \mathbb{M} \ast \dot{\MRB}_{w}$.

Given a condition $(((p, q), r) \in \mathbb{M}^* \ast \dot{\MRB}^*_{w},$ let us identify
$\pi(p)=\pi((p, \emptyset), 1_{\MRB^*_{w}})$ with
\[
(k^{-1})''[p] = p \upharpoonright (\kappa \times \lambda) \cup \{ ((\gamma, \alpha), i): \gamma < \kappa, \alpha \geq \lambda, i \in \{0, 1  \}, ((\gamma, k(\alpha)), i) \in p                  \}.
\]

Let $\MQB_\pi$
be the quotient forcing determined by $\pi$:
\[
\MQB_\pi=\{((p, \dot{q}), \dot{r}) \in \mathbb{M}^* \ast \dot{\MRB}^*_{w}: \pi((p, \dot{q}), \dot{r}) \in g \ast K            \}.
\]
Let us define
\[
\MCB_\pi=\{((p, \emptyset), r):  ((p, \emptyset), r) \in  \MQB_\pi          \}
\]
where the ordering is the one inherited from $\MQB_\pi$,  and let
\[
 \mathbb{T}_\pi =\{(\emptyset, q) \in \mathbb{M}^*: (\emptyset, q ) \in g          \},
\]
with the ordering inherited from $\mathbb{M}^*.$
Also define $\tau_\pi: \MCB_\pi \times \mathbb{T}_\pi \to \MQB_\pi$ by
\[
\tau_\pi(\langle ((p, \emptyset), r), (\emptyset, q)            \rangle)= ((p, q), r).
\]
This is well-defined.
\begin{lemma}
\label{tau-pi is a projection}
$\tau_\pi$ is a projection from $\MCB_\pi \times \mathbb{T}_\pi$ onto $\MQB_\pi$.
\end{lemma}
\begin{proof}
The proof is similar to the proof of Lemma \ref{lem:projection properties}$(a)$.
Clearly $\tau_\pi$ is order preserving. Suppose that
$$(p', \dot q'), \dot r') \leq_{\MQB_\pi} \tau(\langle ((p, \emptyset), \dot r), (\emptyset, \dot q)            \rangle)= ((p, \dot q), \dot r).$$
We are going to find $p^*, \dot q^*$ and $\dot r^*$ such that  $\langle ((p^*, \emptyset), \dot r^*), (\emptyset, \dot q^*)            \rangle \leq_{\MCB_\pi \times \mathbb{T}_\pi} \langle ((p, \emptyset), \dot r), (\emptyset, \dot q)            \rangle$ and
$$\tau(\langle ((p^*, \emptyset), \dot r^*), (\emptyset, \dot q^*)            \rangle)= (p^*, \dot q^*), \dot r^*)  \leq_{\MQB_\pi} (p', \dot q'), \dot r').$$
Let $p^*=p'$. Let $\dot q^*$ be a name such that
\begin{itemize}
\item $p^* \Vdash$``$\dot q^*=\dot q'$''.
\item If $\tilde{p}$ is incompatible with $p'$, then $\tilde{p}\Vdash$``$\dot q^*=\dot q$''.
\end{itemize}
Also set $\dot r^*=\dot r'$. Then $p^*, \dot q^*$ and $\dot r^*$ are as required.
\end{proof}
\begin{lemma}
\label{closure of t-pi}
$\mathbb{T}_\pi$
is $\kappa^+$-closed in $N^1[G_\k \ast g].$
\end{lemma}
\begin{proof}
Similar to the proof of Lemma \ref{lem:projection properties}$(b)$.
\end{proof}
Also, as in the proof of \ref{lem:projection properties}$(c)$, one can show that $\MCB_\pi$ is $\kappa^+$-c.c. in $N^1[G_\k \ast g \ast K].$
Here we prove something stronger, which is needed for the proof of Theorem \ref{tree property at kappa}.
\begin{lemma}
\label{chain condition of c}
$\MCB_\pi \times \MCB_\pi$
is $\kappa^+$-c.c. in $N^1[G_\k \ast g \ast K].$
\end{lemma}
\begin{proof}

Assume towards a contradiction that $A \in N^1[G_\k \ast g \ast K]$ is an antichain in $\MCB_\pi \times \MCB_\pi$
of size $\k^+.$ Let $ \langle (a_i^1, a_i^2): i<\k^+    \rangle$ be an enumeration of $A$, and for $i<\k^+$
and $k \in \{1, 2\}$ let us write $a^k_i$ as
\[
a_i^k= ((p^k_i, \emptyset), \check{d^k_i}^{\frown}\langle w, \lambda^k_i,   \dot{A}^k_i, \dot{H}^k_i, \check{h}^k_i \rangle ).
\]
By shrinking $A$ in necessary, we assume that there is a condition $((p, \dot{q}), \check{d}^{\frown}\langle w, \lambda,   \dot{A}, \dot{H}, \check{h} \rangle ) \in g \ast K$ which forces the following:
\begin{enumerate}
\item $\dot{A}$ is an antichain.

\item There exists $t_1$ such that all $i<\k^+, d^1_i=t_1$.

\item There exists $t_2$ such that all $i<\k^+, d^2_i=t_2$.

\item For some fixed $\eta_1<\k$ and all $i<\k^+, \lambda^1_i=\eta_1$.

\item For some fixed $\eta_2<\k$ and all $i<\k^+, \lambda^2_i=\eta_2$.

\item For some $f_1 \in \Col(\eta_1^{+5}, < \kappa)$ and all $i<\kappa^+, h^1_i=f_1$.

\item For some $f_2 \in \Col(\eta_2^{+5}, < \kappa)$ and all $i<\kappa^+, h^2_i=f_2$.
\end{enumerate}
For each $i$, choose $((p_i, \dot{q}_i), \check{d_i}^{\frown}\langle w, \lambda,   \dot{A}_i, \dot{H}_i, \check{h}_i \rangle) \in g \ast K$
which extends $((p, \dot{q}), \check{d}^{\frown}\langle w, \lambda,   \dot{A}, \dot{H}, \check{h} \rangle )$ and decides both $a^1_i$ and $a^2_i,$ say it forces (for $k \in \{1, 2\}$)
\[
a_i^k= ((p^k_i, \emptyset), \check{t_1}^{\frown}\langle w, \eta_k,   \dot{A}^k_i, \dot{H}^k_i, \check{f}_k \rangle ).
\]
By further shrinking and extending the conditions, we may assume that for some $s$
and all $i<\k^+, d=d_i=s.$

Let
$$((p^{*1}_i, \dot{q}^{*1}_i), \check{s}^{\frown}\langle w, \lambda,   \dot{A}^{*1}_i, \dot{H}^{*1}_i, \check{h}^{*1}_i \rangle)$$
extends $a_i^1$, $((p_i, \dot{q}_i), \check{d_i}^{\frown}\langle w, \lambda,   \dot{A}_i, \dot{H}_i,  \check{h}_i \rangle)$
and $((p, \dot{q}), \check{d}^{\frown}\langle w, \lambda,   \dot{A}, \dot{H}, \check{h} \rangle )$
and such that $\pi(p^{*1}_i)=(k^{-1})''[p^{*1}_i]$ is in the Cohen part of $g \ast K.$
Similarly let
$$((p^{*2}_i, \dot{q}^{*2}_i), \check{s}^{\frown}\langle w, \lambda,   \dot{A}^{*2}_i, \dot{H}^{*2}_i, \check{h}^{*2}_i \rangle)$$
extends $a_i^2$, $((p_i, \dot{q}_i), \check{d_i}^{\frown}\langle w, \lambda,   \dot{A}_i, \dot{H}_i, \check{h}_i \rangle)$
and $((p, \dot{q}), \check{d}^{\frown}\langle w, \lambda,   \dot{A}, \dot{H}, \check{h} \rangle )$
and such that $\pi(p^{*2}_i)=(k^{-1})''[p^{*2}_i]$ is in the Cohen part of $g \ast K.$
Note that in particular $\pi(p^{*1}_i) \| \pi(p^{*2}_i)$ (~$\pi(p^{*1}_i)$ is compatible with $\pi(p^{*2}_i)$).

By $\Delta$-system arguments, we can find $i<j$ such that
$p^{*1}_i \| p^{*1}_j$ and $p^{*2}_i \| p^{*2}_j$.
Let
\[
g^1= ((p^{*1}_i, \dot{q}^{*1}_i), \check{s}^{\frown}\langle w, \lambda,   \dot{A}^{*1}_i, \dot{H}^{*1}_i, \check{h}^{*1}_i,  \rangle) \wedge ((p^{*1}_j, \dot{q}^{*1}_j), \check{s}^{\frown}\langle w, \lambda,   \dot{A}^{*1}_j, \dot{H}^{*1}_j , \check{h}^{*1}_j \rangle)
\]
be the greatest lower bound of $$((p^{*1}_i, \dot{q}^{*1}_i), \check{s}^{\frown}\langle w, \lambda,   \dot{A}^{*1}_i, \dot{H}^{*1}_i, \check{h}^{*1}_i \rangle) $$
and $$ ((p^{*1}_j, \dot{q}^{*1}_j), \check{s}^{\frown}\langle w, \lambda,   \dot{A}^{*1}_j, \dot{H}^{*1}_j, \check{h}^{*1}_j \rangle).$$
Similarly let
\[
g^2= ((p^{*2}_i, \dot{q}^{*2}_i), \check{s}^{\frown}\langle w, \lambda,   \dot{A}^{*2}_i, \dot{H}^{*2}_i, \check{h}^{*2}_i \rangle) \wedge ((p^{*2}_j, \dot{q}^{*2}_j), \check{s}^{\frown}\langle w, \lambda,   \dot{A}^{*2}_j, \dot{H}^{*2}_j, \check{h}^{*2}_j \rangle).
\]
Let
\[
p'= \pi(p^{*1}_i) \cup \pi(p^{*1}_j) \cup \pi(p^{*2}_i) \cup \pi(p^{*2}_j),
\]
which is well-defined. Let
\[
g= ((p', \emptyset), \emptyset) \wedge ((p, \dot{q}), \check{d}^{\frown}\langle w, \lambda,   \dot{A},\dot{H}, \check{h} \rangle ) \bigwedge_{l=i, j} ((p_l, \dot{q}_l), \check{d_l}^{\frown}\langle w, \lambda,   \dot{A}_l, \dot{H}_l,  \check{h}_l \rangle),
\]
be the greatest lower bound of the conditions considered. To continue, we need the following two claims:

\begin{claim}
\label{claim1}
Assume that $$r=((p, \dot{q}), \check{d}^{\frown}\langle w, \lambda,   \dot{A}, \dot{H}, \check{h} \rangle )\in \mathbb{M} \ast \dot{\MRB}_{w}$$
and
$$r^*=((p^*, \emptyset), \check{d^*}^{\frown}\langle w, \lambda,   \dot{A}^*, \dot{H}^*, \check{h}^* \rangle ) \in \mathbb{M}^* \ast \dot{\MRB}^*_{w}$$ and the following conditions are satisfies:
\begin{enumerate}
\item $r \leq \pi(r^*)$.

\item Suppose that $d= \langle d_0, \dots, d_{n-1}                \rangle$ and $d^*= \langle d_0^*, \dots, d_{m-1}^*      \rangle$.
Then $n=m$ and for all $k<n, \kappa^{d_k}=\kappa^{d^*_k}$ and $\lambda^{d_k}=\lambda^{d^*_k}$.

\item For all $k<n, h^{d_k} \leq h^{d^*_k}$.

\item $h \leq h^*.$
\end{enumerate}
Then $r$ does not force
$r^*$  out the quotient $\MCB_\pi$.
\end{claim}
\begin{proof}
Consider the conditions $r$
and $r^*$. The above conditions imply that they are compatible, so let $r \wedge r^*$ be a common extension of them.
Let $\bar g \times \bar K$ be $\mathbb{M}^* \ast \dot{\MRB}^*_{w}$-generic over $V^1[G_\k]$
such that $r \wedge r^* \in \bar g \times \bar K.$ But then $\pi(r) \in \langle \pi''[\bar g \times \bar K]  \rangle$,
the filter on $\mathbb{M} \ast \dot{\MRB}_{w}$ generated by $\pi''[\bar g \times \bar K].$
The result follows immediately.
\end{proof}
\begin{claim}
\label{claim2}
Assume $r$ and $r^*$ are as in Claim \ref{claim1}. Then there exists  $\bar r \leq^* r$ such that $\bar r $ forces ``$r^* \in \MCB_\pi$''.
\end{claim}
\begin{proof}
By Lemma \ref{thm:prikry property}, there exists $\bar r \leq^* r$ which decides ``$r^* \in \MCB_\pi$''. By Claim \ref{claim1}, $\bar r$
cannot force ``$r^* \notin \MCB_\pi$''. So $\bar r \Vdash$``$r^* \in \MCB_\pi$''.
\end{proof}
Note that conditions $g$ and $g^1$ satisfy the conditions in Claim \ref{claim1}, hence by Claim \ref{claim2}, there exists $ \bar g_1 \leq^* g$
which forces ``$g^1 \in \MCB_\pi$''.
Then $\bar g_1$ and $g^2$ satisfy the conditions in Claim \ref{claim1}, so again by Claim \ref{claim2}, there exists $ \bar g_2 \leq^* \bar g_1$
which forces ``$g^2 \in \MCB_\pi$''.
It follows that
\[
\bar g_2 \Vdash \text{~``~} g^1, g^2 \in \MCB_\pi  \text{''}.
\]
But then
\begin{itemize}
\item $\bar g_2 \Vdash$``$(g^1, g^2) \in \MCB_\pi \times \MCB_\pi$''.
\item  $\bar g_2 \Vdash$``$(g^1, g^2) \leq (a_i^1, a_i^2), (a_j^1, a_j^2)$''.
\item $\bar g_2 \leq ((p, \dot{q}), \check{d}^{\frown}\langle \kappa, \lambda,   \dot{A}, \check{f}, \dot{F} \rangle )$.
\end{itemize}
It follows that $\bar g_2 \Vdash$``$\dot{A}$ is an antichain'', and from the above, we get a contradiction.
\end{proof}

We are now ready to complete the proof of Theorem \ref{tree property at kappa}. Note that by our assumption $\dot{T} \in V^1[G_\k]$ is an
 $\mathbb{M} \ast \dot{\MRB}_{w}$ name such that $T=\dot{T}[g \ast K]$
is a $\lambda$-Aronszajn tree in $V^1[G_\k \ast g \ast K].$
Also note that $\dot{T} \in N^1[G_\k].$

By standard arguments, $k(T)_{< \lambda} =T$ and so $T$ has  a cofinal branch in $N^1[G_\k \ast g^* \ast K^*] \subseteq V^1[G_\k \ast g^* \ast K^*].$

Suppose that $X \times Y$ is $\MCB_\pi \times \MQB_\pi$-generic over $V^1[G_\k \ast g \ast K]$
so that $N^1[G_\k \ast g^* \ast K^*] \subseteq N^1[G_\k \ast g \ast K][X \times Y],$ which is possible by Lemma \ref{tau-pi is a projection}.
It follows that $T$ has a cofinal branch in $N[G \ast H \ast K][X \times Y].$
Now lemmas  \ref{lem: preservation product cc}, \ref{lem: preservation product cc-2} \ref{closure of t-pi} and \ref{chain condition of c}  can be used to show that forcing with
$\MCB_\pi \times \MQB_\pi$ over $N^1[G_\k \ast g \ast K]$ does not add cofinal branches to $T$ (see \cite{friedman-honzik2} for details).
We get a contradiction
and Theorem \ref{tree property at kappa} follows.

\subsection{Completing the proof of Theorem \ref{thm:main theorem}}
\label{sec:completing the proof}
In this subsection we complete the proof of Theorem \ref{thm:main theorem}, by showing that in the model
$V^1[G_\k \ast g \ast K]$,~$\TP(\alpha^{++})$ holds for all singular cardinals $\alpha<\kappa$.

Recall that $C=\{\kappa_i: i<\kappa  \}$ is the Radin club added by $K$. Let us also assume that $\min(C)=\aleph_0$.
Recall that $G_\k$ is assumed to be $\MPB_\k$-generic over $V^1$. Let us write it as
\[
G_\k= \langle  \langle G_\a: \a \leq \k      \rangle, \langle G(\a): \a<\k       \rangle\rangle,
\]
which corresponds to the iteration
\[
\MPB_\k= \langle  \langle \MPB_\a: \a \leq \k      \rangle, \langle \dot{\MQB}_\a: \a<\k       \rangle\rangle.
\]
By simple reflecting arguments, we have the following lemma.
\begin{lemma}
\label{defining suitable x}
The set $X \in \mathcal{F}_w,$ where $X$ consists of all those $u \in \mathcal{U}_\infty $ such that $\a=\k_u$ satisfies the following conditions:
\begin{itemize}
\item $\alpha$ is a measurable cardinal.
\item $\MPB_\alpha$ is $\alpha$-c.c. and of size $\alpha.$
\item $\alpha$ remains measurable after forcing with $\MPB_\alpha$ and $\MPB_{\alpha+1}=\MPB_\a \ast \dot{\mathbb{M}}(\alpha, \alpha_*, \alpha_*^+)$.
\item Some  elementary embedding $j: V^1 \to M^1$ with $crit(j)=\alpha$ can be extended to
\[
j: V^1[G_\alpha] \to M^1[j(G_\alpha)]
\]
and then to
\[
j: V^1[G_\alpha \ast G(\alpha)] \to M^1[j(G_\alpha \ast G(\alpha))].
\]
\item $\MPB_{\alpha+1} \Vdash$ `` $\dot{\MPB}_{(\alpha+1, \kappa)}$ does not add any new subsets to $\alpha_*$''.

\end{itemize}
\end{lemma}
\begin{proof}
It suffices to show that $\forall \a<\k^+, w\upharpoonright \a \in j(X)$, which can be easily checked.
\end{proof}
Thus  we can assume that  $$\aleph_0< \alpha \in C \implies \a \in X.$$

On the other hand,  if $\alpha < \kappa$ is a limit cardinal in $V^1[G_\k \ast g \ast K],$ then $\alpha \in \lim(C),$
the set of limit points of $C$,
and $2^\alpha=\alpha^{+3}$.
Thus the following completes the proof:
\begin{theorem}
\label{tree property at limits}
Assume $\alpha \in \lim(C).$ Then $V^1[G_\k \ast g \ast K] \models$``$\TP(\alpha^{++})$''.
\end{theorem}
\begin{proof}
 Fix
$\alpha \in \lim(C)$, and let $\xi<\kappa$ be such that $\alpha = \kappa_{\xi}.$ Note that $\xi$ is a limit ordinal. We have
\[
V^1[G_\k \ast g \ast K] = V^1[G_{\alpha+1}][G_{(\alpha+1, \kappa]}][g][[\vec{u} \upharpoonright \xi, \vec{F} \upharpoonright \xi][\vec{u} \upharpoonright [\xi, \kappa), \vec{F} \upharpoonright [\xi, \kappa)].
\]
and the following hold:
\begin{enumerate}
\item $V^1[G_\k \ast g \ast K]$ is a generic extension of $V^1[G_\k \ast g][\vec{u} \upharpoonright \xi, \vec{F} \upharpoonright \xi]$
by a forcing notion which does not add any new subsets to $\alpha_*$\footnote{Recall that $\alpha_*$ is the least measurable cardinal above $\alpha$.}.

\item Forcing with $\MPB_{(\alpha+1, \kappa]} \ast \dot{\mathbb{M}}$ does not add any subsets to $\alpha_*$; in particular, the forcing notion $\RR_{u_\xi}$ is defined in the same way in the models
$V^1[G_\k \ast g]$ and $V^1[G_{\alpha+1}]$.
\end{enumerate}
It follows that $V^1[G_\k \ast g \ast K]$ is a generic extension of $V^1[G_{\alpha+1}][\vec{u} \upharpoonright \xi, \vec{F} \upharpoonright \xi]$,
by a forcing notion which does not add any new subsets to $\alpha_*.$ Also note that
\begin{enumerate}
\item [(3)] $V^1[G_{\alpha+1}][\vec{u} \upharpoonright \xi, \vec{F} \upharpoonright \xi]\models$``$\alpha^{++}=\alpha_*$''.
 \end{enumerate}
Thus it suffices to prove the following:
\begin{lemma} Tree property at $\alpha_*$
holds in the generic extension  $V^1[G_{\alpha+1}][\vec{u} \upharpoonright \xi, \vec{F} \upharpoonright \xi]$,
which is obtained  using the forcing notion
$$\MPB_{\alpha+1} \ast \dot{\mathbb{R}}_{u_\xi}=\MPB_{\alpha} \ast \dot{\mathbb{M}}_\alpha \ast \dot{\mathbb{R}}_{u_\xi}.$$
\end{lemma}
\begin{proof}
The proof  is very similar to the proof of Theorem \ref{tree property at kappa}.
\end{proof}
 This completes the proof of
Theorem \ref{tree property at limits}.
\end{proof}
Theorem \ref{thm:main theorem} follows.

\section{Tree property at $\aleph_{2n}$'s and $\aleph_{\omega+2}$}
\label{section:local evens}
In this section we prove Theorem \ref{thm:main theorem1}. Thus assume that $GCH$ holds and
 $\eta > \lambda$ are  measurable cardinals above $\k$. We  assume that they are the least such cardinals. Suppose $\kappa$
is an  $H(\eta)$-hypermeasurable cardinal as witnessed by the elementary embedding $j: V \to M \supseteq H(\eta)$.
 We may assume that it is generated by a $(\k, \eta)$-extender.
Let $i: V \to N$ be the ultrapower embedding derived from $j$ and let $k: N \to M$
be such that $j=k \circ i.$

The next lemma can be proved as in Lemma \ref{a very basic iteration lemma}
\begin{lemma}
\label{second very basic iteration lemma}
 Then there exists a cofinality preserving generic extension $V^1$ of $V$ satisfying the following conditions:
\begin{enumerate}
\item [(a)] $V^1 \models$``$GCH$''.
\item [(b)] There is $j^1: V^1 \to M^1$ with critical point $\k$
such that $H(\eta) \subseteq M^1$ and $j^1 \upharpoonright V= j.$

\item [(c)] $j^1$ is generated by a $(\k, \eta)$-extender.

\item [(d)] If $U^1$ is the normal measure derived from $j^1$ and if $i^1: V^1 \to N^1 \simeq \Ult(V^1, U^1)$
is the ultrapower embedding, then there exists  $\bar g \in V^1$
which is $i^1(\Add(\k, \l)_{V^1})$-generic over $N^1.$
Further $i^1 \upharpoonright V=i.$
\end{enumerate}
\end{lemma}
Let $V^1$ be the model constructed above. We need the following lemma which is an analogue of Lemma \ref{preparation model for double succseeor model}.
\begin{lemma}
\label{preparation model for aleph-omega model}
Work in $V^1$.
There exists a forcing iteration $\MPB_\k$ of length $\k$ such that if $G_\k \ast g \ast h$ is $\MPB_\k \ast \dot{\mathbb{M}}(\k,  \l) \ast \dot{\mathbb{M}}(\l,  \eta)$-generic over
$V^1$, then in $V^2=V^1[G_\k \ast g \ast h]$, the following holds:
\begin{enumerate}
\item [(a)] There is $j^2: V^2 \to M^2$ with critical point $\k$
such that $j^2 \upharpoonright V^1= j^1.$

\item [(b)] $j^2$ is generated by a $(\k, \eta)$-extender.

\item [(c)] $V^2 \models$``$\l=\k^{++}+\eta=\k^{+4}+TP(\l)+TP(\eta)$''.

\item [(d)] If $U^2$ is the normal measure derived from $j^2$ and if $i^2: V^2 \to N^2 \simeq \Ult(V^2, U^2)$
is the ultrapower embedding, then there exists  $F \in V^2$
which is $\dot{\mathbb{M}}(\k^{+4},  i^2(\k))_{N^2}$-generic over $N^2.$
\end{enumerate}
\end{lemma}
\begin{proof}
We follow the proof of Lemma \ref{preparation model for double succseeor model}.
For an ordinal $\a \leq \k$ let $\a_*$ and $\a_{**}$
denote the first and second measurable cardinals above $\a.$
Note that $\k_*=\l$ and $\k_{**}=\eta.$

Work in $V^1$. Factor $j^1$ in two steps through the models
\[
N'=\text{~the transitive collapse of~}\{j^1(f)(\k): f: \k \to V^1            \}
\]
\[
 \bar N'=\text{~the transitive collapse of~}\{j^1(f)(\a): f: \k \to V^1, \a < \l            \}.
\]
Again, note that $N'$ is the familiar ultrapower approximating $M^1$, while $\bar N'$ corresponds to the
extender of length $\l$. We have maps

$\hspace{6.cm}$$i': V^1 \to N'$,

$\hspace{6.cm}$$k': N' \to M^1,$

$\hspace{6.cm}$$\bar i': N' \to \bar N',$

$\hspace{6.cm}$$\bar k': \bar N' \to M^1$

such that

$\hspace{5.cm}$$k' \circ i'=j^1$~~~~ $\&$~~~~ $\bar k' \circ \bar i' =k'$.

Let
\[
\MPB_{\k}=\langle  \langle \MPB_\a: \a \leq \k      \rangle, \langle \dot{\MQB}_\a: \a < \k          \rangle\rangle
\]
be the reverse Easton iteration, where
\begin{enumerate}
\item If $\a < \k$ is a measurable limit of measurable cardinals, then $\Vdash_\a$``$\dot{\MQB}_\a =\dot{\mathbb{M}}(\a,  \a_*) \ast \dot{\mathbb{M}}(\a_*,  \a_{**})$''.
\item Otherwise, $\Vdash_\a$``$\dot{\MQB}_\a$ is the trivial forcing''.
\end{enumerate}
Let
$$G_k = \langle \langle G_\a: \a \leq \k      \rangle, \langle  G(\a): \a < \k      \rangle\rangle $$
be
be $\MPB_\k $-generic over $V^1$.

Note that  we can factor
$\MPB_\k \ast \dot{\mathbb{M}}(\k,  \l) \ast \dot{\mathbb{M}}(\l,  \eta)$
as
$$\MPB_\k \ast \dot{\mathbb{M}}(\k,  \l) \ast \dot{\mathbb{M}}(\l,  \eta)=\MPB_\k \ast \Add(\k, \l) \ast \dot{\MQB},$$
where $\dot{\MQB}$ is forced to be $\k^+$-distributive.
So the arguments of the proof of Lemma
\ref{preparation model for double succseeor model} can be used to get the embeddings

$\hspace{3.cm}$$\bar j: V^1[G_\k \ast g \ast h] \to M^1[j^1(G_\k \ast g \ast h)],$

$\hspace{3.cm}$$\bar k': \bar N'[l(G_\k \ast g \ast h)] \to M^1[j^1(G_\k \ast g \ast h)]$,

$\hspace{3.cm}$$l: V^1[G_\k \ast g \ast h] \to \bar N'[ l(G_\k \ast g \ast h)],$

together a filter $F$ which is
 $\mathbb{M}(\k^{+4},  i^2(\k))_{N'[i'(G_\k)]}$ generic over $N'[i'(G_\k)]$.

Let
$$V^2=V^1[G_\k \ast g \ast h],$$
 $$M^2=M^1[j^1(G_\k \ast g \ast h)] $$
and $$N^2=\bar N'[ l(G_\k \ast g \ast h)].$$
Also let $j^2=\bar j.$ We argue
\[
\Ult(V[G_\k \ast g \ast h], U^2) \simeq N^2,
\]
where $U^2$ is the normal measure derived from $j^2$.
To see this, factor $l$ through $l^\dag: V^2 \to N^\dag \simeq \Ult(V^2, U'),$
where $U'$ is the normal measure derived from $l.$ Also let $k^\dag: N^\dag \to N^2.$
Then
$P(\k)_{V^2} \subseteq N^\dag$ and
 $N^\dag \models$``$2^\k=\l$''.
So $\crit(k^\dag)> \l$. Since $\l \subseteq \range(k^\dag)$
and $N^2$ is generated by a $(\k, \l)$-extender, we have $N^2=N^\dag$
and we are done.
So if we let $i^2=l,$ then
\[
i^2: V^2 \to N^2
\]
is the ultrapower embedding.
Finally note that $F$ is generic for the appropriate  ordering.
The lemma follows.
\end{proof}
Also note that $F \in M^2$.
Now, working in $V^2=V^1[G_\k \ast g \ast h]$, we would like to define a version of Prikry forcing.
Set

$\hspace{1.5cm}$ $P^*=\{  f: \kappa \to V^2_\kappa \mid \dom(f) \in U^2$ and $\forall \alpha, f(\alpha) \in    \mathbb{M}(\alpha^{+4}, \k)                \}$.

$\hspace{1.5cm}$ $F^* = \{ f \in P^* \mid i(f)(\kappa) \in F    \}$.

Now define the notion of a constructing pair as in Definition \ref{constructing pair}, where the forcing notions $\Col(\k^{+5}, <i(\k))_N$ and $\Col(\k^{+5}, <j(\k))_M$
are replaced by $\mathbb{M}(\k^{+4}, i(\k))_N$ and  $\mathbb{M}(\k^{+4}, j(\k))_M$
respectively. Then definitions 3.3-3.5 go in the same way.

We now define our  forcing notion as in Section \ref{section:double successor}, but using the different guiding generic filters that we obtained above.
For this aim, and as before, we define forcing notions $\MPB_w, w \in \mathcal{U}_\infty,$ which are the building blocks of our main forcing notion.
 \begin{definition}
If $ w \in \mathcal{U}_\infty$, then $\MPB_w$ is the set of tuples $p= \langle w, \lambda, A, H, h   \rangle$
 such that
 \begin{enumerate}
\item $w$ is a measure sequence.

\item $\lambda < \kappa_w$.

\item $A \in \mathcal{F}_w$.

\item $H \in F^*_w$ with  $\dom(H)=\{ \kappa_v > \lambda \mid v \in A  \}$.

\item $h \in \mathbb{M}(\lambda^{+4},   \kappa_w).$
\end{enumerate}
Note that if $\len(w)=1,$ then the above tuple is of the form  $(w, \lambda, \emptyset, \emptyset, h)$ (where  $\lambda < \kappa_w$ and $h \in \mathbb{M}(\lambda^{+4},   \kappa_w)).$
\end{definition}
The forcing notion $\Rforce_w$ is defined in the same way as before:
\begin{definition}
\label{modified radin forcing definition}
If $w$ is a measure sequence, then $\Rforce_w$ is the set of finite sequences
$$p= \langle  p_k \mid k \leq n     \rangle,$$
where
\begin{enumerate}
\item $p_k = (w_k, \lambda_k, A_k, H_k, h_k) \in \MPB_w$, for each $k \leq n.$

\item $w_n=w.$


\item If $k<n,$ then $\lambda_{k+1}=\kappa_{w_k}$.
\end{enumerate}
\end{definition}
Given $p \in \MRB_{w}$ as above,  $n$ is called the length of $p$ and we denote it by $\len(p).$
The order relations $\leq^*$ and $\leq$ are defined as before.

Now assume that $u$ is the measure sequence constructed using $(j^2, F)$
and set $w=u \upharpoonright 2.$ Let $\Rforce=\Rforce_w$.
Let $K$ be $\Rforce$-generic over $V^1[G_\k \ast g \ast h].$ Let $C$ be the $\omega$-sequence added by $K$
and let $\vec{\k}=\langle  \k_n: n<\omega         \rangle$
enumerate $C$ in the increasing order and note that $\sup_{n<\omega}\k_n=\k.$. We may further assume that  $\k_0=\aleph_0$.  Also let $\vec{F}=\langle  F_n: n<\omega \rangle$
be the $\omega$-sequence added by $K$, where each $F_n$ is $\mathbb{M}(\k_n^{+4},  \k_{n+1})$-generic over $V^1[G_\k \ast g \ast h]$.
The following lemma summarizes the basic properties of $\Rforce$.
\begin{lemma}
\label{properties of forcing R}
\begin{itemize}
\item [(a)] $(\Rforce, \leq)$ satisfies the $\k^+$-c.c.

\item [(b)] Assume $p \in \Rforce$ and $m < n^p$. Then
\[
\Rforce / p \simeq \prod_{i \leq m} \mathbb{M}(\k_i^{+4},  \k_{i+1}) \times \Rforce / p^{> m},
\]
where $p^{> m}= \langle p_{m+1}, \dots, p_{\len(p)}               \rangle$.

\item [(c)] $(\Rforce, \leq, \leq^*)$ satisfies the Prikry property.

\item [(d)] $V^1[G_\k \ast g \ast h \ast K] = V^1[G_\k \ast g \ast h][\vec{F}]$.

\item [(e)] In $V^1[G_\k \ast g \ast h \ast K], \k=\aleph_\omega,~ \lambda=\aleph_{\omega+2}$ and $\eta=\aleph_{\omega+4}$.

\item [(f)] $Card^{V^1[G_\k \ast g \ast h \ast K]} \cap \k = \{\k_n, \k_n^+, \k_n^{++}, \k_n^{+3}, \k_n^{+4}, \k_n^{+5}: n<\omega                 \}.$
\end{itemize}
\end{lemma}
Recall that, given a cardinal $\a \leq \k,$ we are using $\a_*$ to denote the least measurable cardinal above $\a$ and $\a_{**}$ to denote the second measurable cardinal above $\a$;
 so that $\a_{**}=(\a_*)_*$.
The next lemma can be proved by the same arguments as in \cite{friedman-honzik2} (and using Lemma \ref{baby iteration of mitchell}); see also Theorem \ref{tree property at kappa}:
\begin{lemma}
\label{tree property at aleph-omega-2}
In $V^1[G_\k \ast g \ast h \ast K],$ the tree property holds at $\aleph_{\omega+2}.$
\end{lemma}
We now show that the tree property holds at all $\aleph_{2n}$'s, $0<n<\omega.$
The next lemma can be proved by simple reflection arguments.
\begin{lemma}
The set $X \in \mathcal{F}_w,$ where $X$ consists of cardinals $\a < \k$ such that
\begin{enumerate}
\item $\MPB_\alpha$ is $\alpha$-c.c. and of size $\alpha.$
\item $\mathbb{P}_{\a}\Vdash~``~\mathbb{P}(\a)= \mathbb{M}(\a^{+4}, \a_*) \ast \dot{\mathbb{M}}(\a_*^{+4}, \a_{**})  \text{~''} $.
\item $\alpha$ remains measurable after forcing with $\mathbb{P}_\alpha$ and $\mathbb{P}_{\alpha+1}$.
\item Some  elementary embedding $j: V^1 \to M^1$ with $crit(j)=\alpha$ can be extended to
\[
j: V^1[G_\alpha] \to M^1[j(G_\alpha)]
\]
and then to
\[
j: V^1[G_\alpha \ast G(\alpha)] \to M^1[j(G_\alpha \ast G(\alpha))].
\]
\item $\mathbb{P}_{\alpha+1} \Vdash$ `` $\dot{\mathbb{P}}_{(\alpha+1, \kappa)}$ does not add any new subsets to $\alpha_{**}^{+4}$''.

\item $\forall \gamma<\a, \mathbb{P}_{(\gamma, \a]} \times \mathbb{M}(\gamma, \a)\Vdash~``~ \a=\gamma^{++} + TP(\a) \text{~''}.$
\end{enumerate}
\end{lemma}
So we assume that each $\k_n \in X.$

The next lemma follows from Lemma \ref{properties of forcing R}$(f)$.
\begin{lemma}
In $V^1[G_\k \ast g \ast h \ast K],$
\[
\{ \aleph_{2n}: n<\omega               \} = \{ \k_m: m<\omega   \} \cup \{\k_m^{++}: m<\omega      \} \cup \{ \k_m^{+4}: m<\omega        \}.
\]
\end{lemma}
We now prove, in a sequence of lemmas that the tree property holds at all $\aleph_{2n}$'s, $n<\omega$. The case $\k_0=\aleph_0$ follows from K\"{o}nig's theorem stated in the introduction.
We start with the simple case of the tree property at $\k_{m+1}$.
\begin{lemma}
\label{tp at  k-m plus 1}
For each $m,$ $V^1[G_\k \ast g \ast h \ast K] \models$``$\TP(\k_{m+1})$''.
\end{lemma}
\begin{proof}
We can write
\[
V^1[G_\k \ast g \ast h \ast K] = V^1[G_\k \ast g \ast h][\langle F_i: i < m  \rangle][F_m][\langle F_i: m < i < \omega  \rangle].
\]
Working in $V^1[G_\k \ast g \ast h][\langle F_i: i < m  \rangle],$ the filter $F_m$ is generic filter for the forcing notion
$\mathbb{M}(\k_m^{+4},  \k_{m+1}),$ so by Lemma \ref{basic facts on mitchell forcing}$(c)$.
\[
V^1[G_\k \ast g \ast h][\langle F_i: i < m  \rangle][F_m] \models~``~\TP(\k_{m+1})\text{~''.}
\]
But $V^1[G_\k \ast g \ast h \ast K] $ is a generic extension of
$V^1[G_\k \ast g \ast h][\langle F_i: i < m  \rangle][F_m]$ by a forcing notion which does not add new subsets to $\k_{m+1}$,
and so
\[
V^1[G_\k \ast g \ast h \ast K] \models~``~\TP(\k_{m+1})\text{~''.}
\]
\end{proof}
Next we consider cardinals $\k_m^{++}.$
\begin{lemma}
\label{tp at double successor of k-m}
For each $m,$ $V^1[G_\k \ast g \ast h \ast K] \models$``$\TP(\k_m^{++})$''.
\end{lemma}
\begin{proof}
We have
\[
V^1[G_\k \ast g \ast h \ast K] = V^1[G_{\k_m}][G(\k_m)][G_{(\k_m+1, \k])}][g \ast h][ \langle F_i: i < m  \rangle][\langle F_i: m \leq i < \omega  \rangle].
\]
Since the forcing notion $\prod_{i < m} \mathbb{M}(\k_i^{+4},  \k_{i+1})$ is defined in the same way in the models $V^1[G_{\k_m}]$
and $V^1[G_\k \ast g \ast h],$ so
\[
V^1[G_\k \ast g \ast h \ast K] = V^1[G_{\k_m}][ \langle F_i: i < m  \rangle][G(\k_m)][G_{(\k_m+1, \k]}][g \ast h][\langle F_i: m \leq i < \omega  \rangle].
\]
By our convention $\k_m \in X$, $G(\k_m)$ is generic for $\mathbb{M}(\k_m^{+4},  (\k_m)_*) \ast \dot{\mathbb{M}}((\k_m)^{+4}_*,  (\k_m)_{**})$,
so by Lemma \ref{baby iteration of mitchell},
\[
V^1[G_{\k_m}][ \langle F_i: i < m  \rangle][G(\k_m)] \Vdash~``\k_m^{++}=(\k_m)_* ~+~\TP(\k_m^{++}) \text{~''}.
\]
But $V^1[G_\k \ast g \ast h \ast K]$ is a generic extension of $V^1[G_{\k_m}][ \langle F_i: i < m  \rangle][G(\k_m)]$ by a forcing notion which does not add any new subsets
to $(\k_m)_*,$ and so
$V^1[G_\k \ast g \ast h \ast K] \models$``$\TP(\k_m^{++})$''.
\end{proof}
Now we consider the cardinals $\k_m^{+4}.$
\begin{lemma}
\label{tp at fourth successor of k-m}
For each $m,$ $V^1[G_\k \ast g \ast h \ast K] \models$``$\TP(\k_m^{+4})$''.
\end{lemma}
\begin{proof}
As above,
\[
V^1[G_\k \ast g \ast h \ast K] = V[G_{\k_m}][ \langle F_i: i < m  \rangle][G(\k_m)][G_{(\k_m+1, \k]}][g \ast h][F_m][\langle F_i: m < i < \omega  \rangle].
\]
But $\mathbb{M}(\k_m^{+4},  \k_{m+1})$ is defined in the same way in the models $V^1[G_{\k_m}][ \langle F_i: i < m  \rangle][G(\k_m)]$
and $V^1[G_\k \ast g \ast h \ast K],$ so $V^1[G_\k \ast g \ast h \ast K]$ is equal to
\[
V^1[G_{\k_m}][ \langle F_i: i < m  \rangle][G(\k_m)][G_{(\k_m+1, \k_{m+1})} \times F_m][G_{(\k_{m+1}+1, \k]}][g \ast h][\langle F_i: m < i < \omega  \rangle].
\]
Since $V^1[G_\k \ast g \ast h \ast K]$ is obtained from $V^1[G_{\k_m}][ \langle F_i: i < m  \rangle][G(\k_m)][G_{(\k_m+1, \k_{m+1})} \times F_m]$ by a forcing which does not add new subsets to $(\k_m)_{**} (=((\k_m)^{+4})^{V^1[G \ast g \ast h \ast K]})$, it suffices to show $\TP((\k_m)_{**})$ holds in $V^1[G_{\k_m}][ \langle F_i: i < m  \rangle][G(\k_m)][G_{(\k_m+1, \k_{m+1})} \times F_m]$.
Now, the model $V^1[G_{\k_m}][ \langle F_i: i < m  \rangle][G(\k_m)][G_{(\k_m+1, \k_{m+1})} \times F_m]$ is a generic extension of $V^1[G_{\k_m}][ \langle F_i: i < m  \rangle]$  by the forcing notion
\[
\mathbb{P}(\k_m) \ast (\dot{\mathbb{P}}_{(\k_m+1, \k_{m+1})} \times \dot{\mathbb{M}}(\k_m^{+4},  \k_{m+1})),
\]
and by Lemmas \ref{iteration of mitchell} and \ref{laver vs mitchell},
\[
V^1[G_{\k_m}][ \langle F_i: i < m  \rangle][G(\k_m)][G_{(\k_m+1, \k_{m+1})} \times F_m] \models~``~\TP((\k_m)_{**})\text{~''}.
\]
The lemma follows.
\end{proof}

 Theorem \ref{thm:main theorem1} follows.

\section{Tree property at all regular even cardinals}
\label{section:all even}
In this section we prove Theorem \ref{thm:main theorem2}. In Subsection \ref{A new variant of Radin forcing}, we  present some of the basic properties
of the new version of Radin forcing we defined in  Section \ref{section:local evens}.
Then in Subsection \ref{sec: modified final model}, we define the forcing notion needed which is used for the proof of our main theorem.
Finally in Subsection \ref{the tree property holds at all even cardinals below}, we complete the proof of Theorem \ref{thm:main theorem2}.

\subsection{A new variant of Radin forcing}
\label{A new variant of Radin forcing}
Through this subsection, we assume that the following
conditions are satisfied:
\begin{itemize}
\item $\kappa$ is an $H(\k^{++})$-hypermeasurable cardinal, $2^\k=\k^{++}$ and $2^{\k^{++}}=\k^{+4}$.

\item  There is $j: V \to M$ with critical point $\k$
such that $H(\k^{++}) \subseteq M$.

\item  $j$ is generated by a $(\k, \k^{+5})$-extender.

\item If $U$ is the normal measure derived from $j$ and if $i: V \to N \simeq \Ult(V, U)$
is the ultrapower embedding, then there exists  $F \in V$
which is $\mathbb{M}(\kappa^{+4},  i(\k))_{N}$-generic over $N.$
\end{itemize}
Let $\Rforce_w$ be the modified version of Radin forcing that we defined in Section \ref{section:local evens}, and let us  review it basic properties in the general context. The next lemma can be proved as in Lemma
\ref{chain condition for radin forcing}
\begin{lemma}
\label{chain condition for modified radin forcing}
$\Rforce_{w}$ satisfies  the $\kappa_w^+$-chain condition.
\end{lemma}
The following is an analogue of the factorization lemma \ref{thm:factorization lemma}.
\begin{lemma} (The factorization lemma)  \label{thm:modified factorization lemma}
Suppose that $p=\langle p_0,\dots,p_n \rangle\in \Rforce_{w}$ where
$p_i=\langle \kappa_i,\lambda_i,  A_i, f_i,F_i \rangle$ and $m<n.$ Set $p^{\leq m}= \langle p_0, \dots, p_m \rangle$
and $p^{>m} = \langle  p_{m+1}, \dots, p_n     \rangle$.

\begin{enumerate}
\item [(a)]  $p^{\leq m} \in \Rforce_{w\upharpoonright \kappa_m+1},$
$p^{>m} \in \Rforce_{w}$ and there exists
$$i: \Rforce_{w} / p \rightarrow \Rforce_{w\upharpoonright \kappa_m+1}/p^{\leq m} \times \Rforce_{w} / p^{>m}$$
which is an isomorphism with respect to both $\leq^*$ and $\leq.$

\item [(b)] If $m+1 <n,$ then there exists
$$i: \Rforce_{w} / p \rightarrow \Rforce_{w \upharpoonright \kappa_m+1}/p^{\leq m} \times \mathbb{M}(\kappa_m^{+4},  \kappa_{m+1}) \times \Rforce_{w} / p^{>m+1}$$
which is an isomorphism with respect to both $\leq^*$ and $\leq.$ \hfill$\Box$
\end{enumerate}
\end{lemma}
The following can be proved as before:
\begin{lemma}  \label{thm:modified prikry property}
\begin{enumerate}
\item [(a)] $(\RR_{w}, \leq, \leq^*)$ satisfies the Prikry property.

\item [(b)] Assume $\len(w)=\k_w^+$ Then forcing with $\MRB_{w}$ preserves the inaccessibility of $\kappa_w$.
\end{enumerate}
\end{lemma}
From now on assume that $\len(w)=\k^+$.
Suppose $K \subseteq \Rforce_{w}$ is generic over $V$ and define the club $C$ and the sequence $\vec{u}= \langle u_\xi: \xi <\k  \rangle$
and $\vec{\k}= \langle \k_\xi: \xi < \k      \rangle$ as before. Let the sequence
$\vec{F}=\langle F_\xi: \xi<\k \rangle$ be such that each $F_\xi$ is $\mathbb{M}(\k_\xi^{+4},  \k_{\xi+1})$-generic over $V$, which is produced by $K.$
The next lemma can be proved as in Lemma \ref{thm: bounding sets}.
 \begin{lemma}\label{thm: modified bounding sets}
 \begin{enumerate}
\item [(a)] $V[K]=V[\vec{u}, \vec{F}].$

\item [(b)] For every limit ordinal $\xi<\kappa, \langle \vec{u}\upharpoonright \xi, \vec{F}\upharpoonright \xi  \rangle$  is $\RR_{u_\xi}$-generic over $V$,
and $\langle \vec{u} \upharpoonright [\xi, \kappa), \vec{F}\upharpoonright$
 $[\xi, \kappa) \rangle$ is $\RR_w$-generic over $V[\vec{u}\upharpoonright \xi, \vec{F}\upharpoonright \xi].$

\item [(c)] For every $\gamma <\kappa$ and every $A \subseteq \gamma$ with $A\in V[\vec{u}, \vec{F}],$ we have $A\in V[\vec{u}\upharpoonright \xi, \vec{F}\upharpoonright \xi],$
 where $\xi$ is the least ordinal such that $\gamma< \kappa_\xi.$
\end{enumerate}
\end{lemma}

\subsection{The final model}
\label{sec: modified final model}
In this subsection we define the final model we are going to work with.
Thus assume that $GCH$ holds,
 $\eta > \lambda$ are  measurable cardinals above $\k$. We  assume that they are the least such cardinals. Suppose $\kappa$
is an  $H(\eta^+)$-hypermeasurable cardinal. Let $j: V \to M \supseteq H(\eta^+)$
witness this. We may assume that it is generated by a $(\k, \eta^+)$-extender.
Let $i: V \to N$ be the ultrapower embedding derived from $j$ and let $k: N \to M$
be such that $j=k \circ i.$

The next lemma can be proved as in Lemma \ref{a very basic iteration lemma}
\begin{lemma}
\label{third very basic iteration lemma}
 Then there exists a cofinality preserving generic extension $V^1$ of $V$ satisfying the following conditions:
\begin{enumerate}
\item [(a)] $V^1 \models$``$GCH$''.
\item [(b)] There is $j^1: V^1 \to M^1$ with critical point $\k$
such that $H(\eta^+) \subseteq M^1$ and $j^1 \upharpoonright V= j.$

\item [(c)] $j^1$ is generated by a $(\k, \eta^+)$-extender.

\item [(d)] If $U^1$ is the normal measure derived from $j^1$ and if $i^1: V^1 \to N^1 \simeq \Ult(V^1, U^1)$
is the ultrapower embedding, then there exists  $\bar g \in V^1$
which is $i^1(\Add(\k, \l)_{V^1})$-generic over $N^1.$
Further $i^1 \upharpoonright V=i.$
\end{enumerate}
\end{lemma}
\begin{lemma}
\label{preparation model for all evens}
Work in $V^1$.
There exists a forcing iteration $\MPB_\k$ of length $\k$ such that if $G_\k \ast g \ast h$ is $\MPB_\k \ast \dot{\mathbb{M}}(\k,  \l) \ast \dot{\mathbb{M}}(\l,  \eta)$-generic over
$V^1$, then in $V^2=V^1[G_\k \ast g \ast h]$, the following holds:
\begin{enumerate}
\item [(a)] $V^2 \models$``$\l=\k^{++}+\eta=\k^{+4}+TP(\l)+TP(\eta)$''.

\item [(b)] There is $j^2: V^2 \to M^2$ with critical point $\k$
and $H(\k^{++}) \subseteq M^2$
such that $j^2 \upharpoonright V^1= j^1.$

\item [(c)] $j^2$ is generated by a $(\k, \eta^+)$-extender.

\item [(d)] If $U^2$ is the normal measure derived from $j^2$ and if $i^2: V^2 \to N^2 \simeq \Ult(V^2, U^2)$
is the ultrapower embedding, then there exists  $F \in V^2$
which is $\dot{\mathbb{M}}(\k^{+4},  i^2(\k))_{N^2}$-generic over $N^2.$
\end{enumerate}
\end{lemma}
\begin{proof}
The model $V^2=V^1[G_\k \ast g \ast h]$ constructed in the proof of Lemma \ref{preparation model for aleph-omega model}
does the job. The additional assumption of $\k$ being $H(\eta^+)$-hypermeasurable
guarantees that
$H(\k^{++}) \subseteq M^2$.
\end{proof}

In particular, note that in the model $V^2$, the hypotheses
at the beginning of Subsection \ref{A new variant of Radin forcing}
are satisfied; so we can consider the forcing notion $\MRB_w,$
where  $w= u \upharpoonright \k^+$ and $u$ is constructed using the  pair $(j^2, F)$.
Let $K$ be $\MRB_{w}$-generic over $V^2.$
Build the sequences $\vec{\k}= \langle \k_\xi: \xi < \k \rangle, \vec{u} = \langle  u_\xi: \xi < \k \rangle$ and $\vec{F} = \langle F_\xi: \xi<\kappa \rangle$ from $K$, as before.

\subsection{In $V^1[G_\k \ast g \ast h \ast K]$,  the tree property holds at all regular even cardinals below $\kappa$}
\label{the tree property holds at all even cardinals below}
Here we complete the proof of Theorem \ref{thm:main theorem2}.
As before, given a cardinal $\a \leq \k,$ let $\a_*$ denote the least measurable cardinal above $\a$ and let $\a_{**}$ denote the second measurable cardinal above $\a$.
 Now note that
\[
Card^{V^1[G_\k \ast g \ast h \ast K]} \cap \kappa = \{\k_\xi, \k_\xi^+: \xi<\k                    \} \cup \{(\k_\xi)_*, (\k_\xi)_*^+: \xi < \k     \} \cup  \{(\k_\xi)_{**}, (\k_\xi)_{**}^+: \xi < \k     \},
\]
Also note that
 if $\alpha < \kappa$ is a singular cardinal in $V^1[G_\k \ast g \ast h \ast K]$, then $\alpha \in \lim(C)$, i.e., $\a=\k_\xi$ for some limit ordinal $\xi<\k$.
 The following lemma is immediate:
\begin{lemma}
 In $V^1[G_\k \ast g \ast h \ast K],$ the set of uncountable regular even cardinals below $\k$ is equal to
 \[
 \{ \k_\xi^{++}: \xi < \k     \} \cup \{\k_\xi^{+4}: \xi < \k          \} \cup \{  \k_{\xi+1}: \xi < \k \}.
 \]
 \end{lemma}

Before we continue, let us show that we can choose the cardinals $\k_\xi$ in a suitable way, which is guaranteed by the following lemma, which is an analogue
 of Lemma \ref{defining suitable x}.
\begin{lemma}
\label{suitable choice of elements}
The set $X \in \mathcal{F}_w,$ where $X$ consists of all those $u \in \mathcal{U}_\infty $ such that $\a=\k_u$ satisfies the following conditions:
\begin{enumerate}
\item $\mathbb{P}_\alpha$ is $\alpha$-c.c. and of size $\alpha.$
\item $\mathbb{P}_{\a}\Vdash~``~\mathbb{P}(\a)= \mathbb{M}(\a, \a_*) \ast \dot{\mathbb{M}}(\a_*, \a_{**})  \text{~''} $.
\item $\alpha$ remains measurable after forcing with $\mathbb{P}_\alpha$ and $\mathbb{P}_{\alpha+1}$.
\item Some  elementary embedding $j: V^1 \to M^1$ with $crit(j)=\alpha$ can be extended to
\[
j: V^1[G_\alpha] \to M^1[j(G_\alpha)]
\]
and then to
\[
j: V^1[G_\alpha \ast G(\alpha)] \to M^1[j(G_\alpha \ast G(\alpha))].
\]
\item $\mathbb{P}_{\alpha+1} \Vdash$ `` $\dot{\mathbb{P}}_{(\alpha+1, \kappa)}$ does not add any new subsets to $\alpha_{**}^{+4}$''.

\item $\forall \gamma<\a, \mathbb{P}_{(\gamma, \a]} \times \mathbb{M}(\gamma, \a)\Vdash~``~ \a=\gamma^{++} + TP(\a) \text{~''}.$
\end{enumerate}
\end{lemma}

The next lemma can be proved as in Theorems \ref{tree property at kappa} and \ref{tree property at limits},
combined with ides of the proof of Lemma \ref{tp at double successor of k-m}.
\begin{lemma}
\label{tree property at limits below kappa}
$V^1[G_\k \ast g \ast h \ast K] \models$``For all limit ordinals $\xi<\k, \k_\xi^{++}=(\k_{\xi})_*$ and $\TP((\k_\xi)_*)$ holds''.
\end{lemma}

Before we continue, let us make a simple remark. Assume
 $\bar \xi$ is a limit ordinal.
Then we can write $V^1[G_\k \ast g \ast h \ast K]$ as
\[
V[G_\k \ast g \ast h][\vec{u} \upharpoonright \bar \xi, \vec{F} \upharpoonright \bar \xi][\vec{u} \upharpoonright [\bar \xi, \bar \xi + \omega), \vec{F} \upharpoonright [\bar \xi, \bar \xi + \omega)][\vec{u} \upharpoonright [\bar \xi+\omega, \k), \vec{F} \upharpoonright [\bar \xi+\omega, \k)].
\]
On the other hand:
\begin{enumerate}
\item  $V^1[G_\k \ast g \ast h \ast K]$ is a generic extension of $V^1[G_\k \ast g \ast h][\vec{u} \upharpoonright \bar \xi, \vec{F} \upharpoonright \bar \xi][\vec{u} \upharpoonright [\bar \xi, \bar \xi + \omega), \vec{F} \upharpoonright [\bar \xi, \bar \xi + \omega)]$ by a forcing notion which does not add any new subsets
    to $\k_{\bar \xi+\omega}$.
\item By standard arguments, $V^1[G_\k \ast g \ast h][\vec{u} \upharpoonright \bar \xi, \vec{F} \upharpoonright \bar \xi][\vec{u} \upharpoonright [\bar \xi, \bar \xi + \omega), \vec{F} \upharpoonright [\bar \xi, \bar \xi + \omega)]$ is generic extension of
    $V^1[G_\k \ast g \ast h ][\vec{u} \upharpoonright \bar \xi, \vec{F} \upharpoonright \bar \xi]$ by a forcing notion which is forcing equivalent to the forcing notion $\Rforce$
    of Section \ref{section:local evens} (for suitable choices of the normal measure and guiding generic filters).
\end{enumerate}
 So, given any limit ordinal $\bar \xi,$ we can use the arguments of Section \ref{section:local evens}
 to conclude that the model $V^1[G_\k \ast g \ast h][\vec{u} \upharpoonright \bar \xi, \vec{F} \upharpoonright \bar \xi][\vec{u} \upharpoonright [\bar \xi, \bar \xi + \omega), \vec{F} \upharpoonright [\bar \xi, \bar \xi + \omega)]$ satisfies:
 \[
  ~``~\forall n<\omega,~ \TP(\k_{\bar \xi+n}^{++})+\TP(\k_{\bar \xi+n}^{+4})+\TP(\k_{\bar \xi+\omega+2}) \text{~''.}
 \]
 Using the above remark, and by the same arguments as in Sections \ref{section:double successor} and \ref{section:local evens}
 one can prove Theorem \ref{thm:main theorem2}. Below we present more details for completeness.
 First we prove an analogue of Lemma \ref{tree property at limits below kappa} for successor ordinals.
\begin{lemma}
\label{tree property at successor k-i}
$V^1[G_\k \ast g \ast h \ast K] \models$``For all successor ordinals $\xi<\k, \k_\xi^{++}=(\k_\xi)_*$ and the tree property at $(\k_\xi)_*$  holds''.
\end{lemma}
\begin{proof}
Suppose $\xi=\zeta+1$.
The model $V^1[G_\k \ast g \ast h \ast K]$ is an extension of
$V^1[G_\k \ast g \ast h][\vec{u} \upharpoonright \xi, \vec{F} \upharpoonright \xi]$
by a forcing notion which does not add new subsets to $\k_\xi^{++}=(\k_\xi)_*$; so it suffices to show that $\TP((\k_\xi)_*)$
holds in $V^1[G_\k \ast g \ast h ][\vec{u} \upharpoonright \xi, \vec{F} \upharpoonright \xi]$.
But this last model is equal to
\[
V^1[G_{\k_\xi}][G(\k_\xi)][G_{(\k_\xi+1, \k]}][g \ast h][\vec{u} \upharpoonright \zeta, \vec{F} \upharpoonright \zeta][F_\zeta].
\]
Since the forcing notion $\Rforce_{u_\zeta}$ is defined in the same way in the models $V^1[G_\k \ast g \ast h]$
and $V^1[G_{\k_\xi}],$ so $V^1[G_\k \ast g \ast h][\vec{u} \upharpoonright \xi, \vec{F} \upharpoonright \xi]$ equals
\[
V^1[G_{\k_\xi}][\vec{u} \upharpoonright \zeta, \vec{F} \upharpoonright \zeta][G(\k_\xi)][G_{(\k_\xi+1, \k]}][g \ast h][F_\zeta]=V[G_{\k_\xi}][\vec{u} \upharpoonright \zeta, \vec{F} \upharpoonright \zeta][F_\zeta][G(\k_\xi)][G_{(\k_\xi+1, \k]}][g \ast h].
\]
It follows that  $V^1[G_\k \ast g \ast h][\vec{u} \upharpoonright \xi, \vec{F} \upharpoonright \xi]$
is an extension of $V^1[G_{\k_\xi}][\vec{u} \upharpoonright \zeta, \vec{F} \upharpoonright \zeta][F_\zeta][G(\k_\xi)]$
by a forcing which does not add new subsets to $(\k_\xi)_*$, so we just need to show that
$\TP((\k_\xi)_*)$
holds in $V^1[G_{\k_\xi}][\vec{u} \upharpoonright \zeta, \vec{F} \upharpoonright \zeta][F_\zeta][G(\k_\xi)]$.

But $G(\k_\xi)$ is generic for  $\mathbb{M}(\k_\xi, (\k_\xi)_*) \ast \dot{\mathbb{M}}((\k_\xi)_*, (\k_{\xi})_{**})$
and so by Lemma \ref{baby iteration of mitchell},
\[
V^1[G_{\k_\xi}][\vec{u} \upharpoonright \zeta, \vec{F} \upharpoonright \zeta][F_\zeta][G(\k_\xi)] \models~``~\TP((\k_\xi)_*)\text{~''.}
\]
The lemma follows.
\end{proof}
Next we prove the following:
\begin{lemma}
\label{tree property at double successor k-i}
$V^1[G_\k \ast g \ast h \ast K] \models$``For all  ordinals $\xi<\k, ~\TP(\k_{\xi+1})$  holds''.
\end{lemma}
\begin{proof}
We have
\[
V^1[G_\k \ast g \ast h \ast K]= V^1[G_\k \ast g \ast h][\vec{u} \upharpoonright \xi, \vec{F} \upharpoonright \xi][F_\xi][\vec{u} \upharpoonright (\xi+1, \k), \vec{F} \upharpoonright (\xi+1, \k))]
\]
Now $F_\xi$ is a generic filter for $\mathbb{M}(\k_\xi^{+4}, \k_{\xi+1})$ and by Lemma \ref{basic facts on mitchell forcing}$(c)$,
we have
\begin{center}
$V^1[G_\k \ast g \ast h][\vec{u} \upharpoonright \xi, \vec{F} \upharpoonright \xi][F_\xi]\models$``$\TP(\k_{\xi+1})$''.
\end{center}
On the other hand, the models $V^1[G_\k \ast g \ast h \ast K]$ and
$V^1[G_\k \ast g \ast h][\vec{u} \upharpoonright \xi, \vec{F} \upharpoonright \xi][F_\xi]$ have the same subsets of $\k_{\xi+1}$,
and hence
$V^1[G_\k \ast g \ast h \ast K]\models$``$\TP(\k_{\xi+1})$''.
\end{proof}

\begin{theorem}
\label{tree property at double fourth successor of limits below kappa}
$V^1[G_\k \ast g \ast h \ast K] \models$``For all  ordinals $\xi<\k, \k_\xi^{+4}=(\k_\xi)_{**}$ and the tree property at $(\k_\xi)_{**}$ holds''.
\end{theorem}
\begin{proof}
By similar analysis as above, it suffices to show that $\TP((\k_\xi)_{**})$ holds in the model
\[
V^1[G_{\k_\xi}][\vec{u} \upharpoonright \xi, \vec{F} \upharpoonright \xi][G(\k_\xi)][G_{(\k_\xi+1, \k_{\xi+1}]}][F_\xi],
\]
which is an extension of $V^1[G_{\k_\xi}][\vec{u} \upharpoonright \xi, \vec{F} \upharpoonright \xi]$
by the forcing notion
$$\mathbb{M}(\k_\xi, (\k_\xi)_*) \ast \dot{\mathbb{M}}((\k_\xi)_*, (\k_{\xi})_{**}) \ast (\dot{\mathbb{L}}_{(\k_\xi+1, \k_{\xi+1}]} \times \dot{\mathbb{M}}((\k_\xi)_{**}, \k_{\xi+1})).$$
So, by lemmas by Lemmas \ref{iteration of mitchell} and \ref{laver vs mitchell},
\[
V^1[G_{\k_\xi}][\vec{u} \upharpoonright \xi, \vec{F} \upharpoonright \xi][G(\k_\xi)][G_{(\k_\xi+1, \k_{\xi+1}]}][F_\xi] \models ~``~\TP((\k_\xi)_{**})\text{~''.}
\]
The lemma follows.
\end{proof}
 Theorem \ref{thm:main theorem2} follows.

School of Mathematics, Institute for Research in Fundamental Sciences (IPM), P.O. Box:
19395-5746, Tehran-Iran.

E-mail address: golshani.m@gmail.com


\begin{thebibliography}{99}
\bibitem{cummings} Cummings, James, A model in which GCH holds at successors but fails at limits. Trans. Amer. Math. Soc. 329 (1992), no. 1, 1-–39.

\bibitem{cummings1} Cummings, James. Iterated forcing and elementary embeddings. Handbook of set theory. Vols. 1, 2, 3, 775--883, Springer, Dordrecht, 2010.


\bibitem{cummings-foreman} Cummings, James; Foreman, Matthew; The tree property. Adv. Math. 133 (1998), no. 1, 1–32.

\bibitem{five authors} Cummings, James; Friedman, Sy-David; Magidor, Menachem; Rinot, Assaf; Sinapova, Dima,  The eightfold way, submitted.

\bibitem{friedman-halilovic} Friedman, Sy-David; Halilovic, Ajdin The tree property at $\aleph_{\omega+2}$. J. Symbolic Logic 76 (2011), no. 2, 477-�490.


\bibitem{friedman-honzik0} Friedman, Sy-David; Honzik, Radek The tree property at the $\aleph_{2n}$'s and the failure of SCH at $\aleph_\omega$. Ann. Pure Appl. Logic 166 (2015), no. 4, 526–-552.

\bibitem{friedman-honzik1} Friedman, Sy-David, Honzik, Radek;  Stejskalova, Sarka,   The tree property at the double successor of a singular cardinal with a larger gap, submitted.
\bibitem{friedman-honzik2} Friedman, Sy-David, Honzik, Radek;  Stejskalova, Sarka,  The tree property at $\alpha_{\omega+2}$ with a finite gap, submitted.


\bibitem{gitik} Gitik, Moti; Prikry-type forcings. Handbook of set theory. Vols. 1, 2, 3, 1351–-1447, Springer, Dordrecht, 2010.

\bibitem{gitik-tree} Gitik, Moti, Extender based forcings, fresh sets and Aronszajn trees. Infinity, computability, and metamathematics, 183-–203, Tributes, 23, Coll. Publ., London, 2014.
\bibitem{golshani-diamond} Golshani, Mohammad, (Weak) diamond can fail at the least inaccessible cardinal, submitted.

\bibitem{golshani-mitchell}, Golshani, Mohammad; Mitchell, William,  On a question of Hamkins and L\"{o}we on the modal logic of collapse forcing, submitted.


\bibitem{golshani-mohammadpour} Golshani, Mohammad; Mohammadpour, Rahman; Tree property at double successor of a singular strong limit cardinal of uncountable cofinality.


\bibitem{hayut-eskew} Hayut, Yair, Eskew, Monroe, On the consistency of local and global versions of Chang's conjecture, submitted.




\bibitem{jech} Jech, Thomas Set theory. The third millennium edition, revised and expanded. Springer Monographs in Mathematics. Springer-Verlag, Berlin, 2003. xiv+769 pp. ISBN: 3-540-44085-2

 \bibitem{kanamori} Kanamori, Akihiro, The higher infinite. Large cardinals in set theory from their beginnings. Second edition. Springer Monographs in Mathematics. Springer-Verlag, Berlin, 2003. xxii+536 pp.

\bibitem{KunenTall}  Kunen, Kenneth; Tall, Franklin D., Between Martin's axiom and Souslin's hypothesis. Fund. Math. 102 (1979), no. 3, 173--181.

\bibitem{laver} Laver, Richard, Making the supercompactness of $\k$ indestructible under $\k$-directed closed forcing. Israel J. Math. 29 (1978), no. 4, 385-–388.

\bibitem{levy-solovay} Lévy, A.; Solovay, R. M., Measurable cardinals and the continuum hypothesis. Israel J. Math. 5 1967 234–-248.

\bibitem{MagidorShelah96} Magidor, Menachem; Shelah, Saharon, The tree property at successors of singular cardinals. Arch. Math. Logic 35 (1996), no. 5--6, 385-–404.


\bibitem{merimovich1} Merimovich, Carmi; Extender-based Radin forcing. Trans. Amer. Math. Soc. 355 (2003), no. 5, 1729-–1772.


\bibitem{merimovich2} Merimovich, Carmi; A power function with a fixed finite gap everywhere. J. Symbolic Logic 72 (2007), no. 2, 361-–417.

\bibitem{mitchell72} Mitchell, William, Aronszajn trees and the independence of the transfer property. Ann. Math. Logic 5 (1972/73), 21-–46.

 \bibitem{mitchell82} Mitchell, William, How weak is a closed unbounded ultrafilter? Logic Colloquium '80 (Prague, 1980), pp. 209-–230, Stud. Logic Foundations Math., 108, North-Holland, Amsterdam-New York, 1982.



\bibitem{neeman} Neeman, Itay, The tree property up to $\aleph_{\omega+1}$. J. Symb. Log.  79  (2014),  no. 2, 429-–459.

\bibitem{radin} Radin, Lon Berk,  Adding closed cofinal sequences to large cardinals. Ann. Math. Logic 22 (1982), no. 3, 243–-261.

\bibitem{unger} Unger, Spencer, Fragility and indestructibility of the tree property. Arch. Math. Logic 51 (2012), no.
5-6, 635--645.


\bibitem{unger1} Unger, Spencer, Iterating along a Prikry sequence. Fund. Math. 232 (2016), no. 2, 151–-165.
\end{thebibliography}
\end{document}